\documentclass{amsart}

\usepackage{amscd}
\usepackage{amssymb}

\usepackage{color}

\newtheorem{lem}[equation]{Lemma}
\newtheorem{lemma}[equation]{Lemma}

\newtheorem{prop}[equation]{Proposition}
\newtheorem{conj}[equation]{Conjecture}

\newtheorem{specthm}{Theorem}

\newtheorem*{thm*}{Theorem}
\newtheorem*{prop*}{Proposition}
\newtheorem*{cor*}{Corollary}
\newtheorem*{lem*}{Lemma}
\newtheorem*{MT*}{Main Theorem}

\newtheorem*{ques*}{Question}

\theoremstyle{definition} %

\newtheorem*{defn*}{Definition}

\newtheorem{eg}[equation]{Example}

\theoremstyle{remark} %
\newtheorem{rmk}[equation]{Remark}

\newtheorem*{rmk*}{Remark}
\newtheorem*{rmks*}{Remarks}

\DeclareMathOperator{\rank}{rank}
\DeclareMathOperator{\car}{char}
\DeclareMathOperator{\Lie}{Lie}

\DeclareMathOperator{\im}{im}

\renewcommand{\P}{\mathbb{P}}
\newcommand{\C}{\mathbb{C}}
\newcommand{\Z}{\mathbb{Z}}
\newcommand{\Q}{\mathbb{Q}}

\newcommand{\lie}{\mathfrak{g}}
\newcommand{\lsub}{\mathfrak{h}}

\newcommand{\qform}[1]{{\left\langle{#1}\right\rangle}}                   % a quadratic form
\DeclareMathOperator{\PGL}{PGL}
\DeclareMathOperator{\SL}{SL}
\DeclareMathOperator{\Sp}{Sp}
\DeclareMathOperator{\PSp}{PSp}

\DeclareMathOperator{\SO}{SO}

\DeclareMathOperator{\GL}{GL}
\DeclareMathOperator{\Ad}{Ad}

\DeclareMathOperator{\Spin}{Spin}

\DeclareMathOperator{\Tran}{Tran}

\newcommand{\F}{\mathbb{F}}

\newcommand{\g}{\mathfrak{g}}
\newcommand{\gt}{\tilde{\g}}

\newcommand{\gl}{\mathfrak{gl}}
\renewcommand{\sl}{\mathfrak{sl}}
\newcommand{\pgl}{\mathfrak{pgl}}
\renewcommand{\sp}{\mathfrak{sp}}
\newcommand{\spin}{\mathfrak{spin}}
\newcommand{\so}{\mathfrak{so}}
\newcommand{\go}{\mathfrak{go}}

\newcommand{\tor}{\mathfrak{t}}

\DeclareMathOperator{\Sym}{S}

\DeclareMathOperator{\Spec}{Spec}

\newcommand{\z}{\mathfrak{z}}

\newcommand{\Gm}{\mathbb{G}_m}

\newcommand{\drho}{\mathrm{d}\rho}
\newcommand{\dpi}{\mathrm{d}\pi}
\newcommand{\dq}{\mathrm{d}q}
\newcommand{\dchi}{\mathrm{d}\chi}
\newcommand{\dtau}{\mathrm{d}\tau}

\newcommand{\la}{\lambda}
\newcommand{\ot}{\otimes}

\newcommand{\eps}{\varepsilon}

\newcommand{\Gt}{\widetilde{G}}

\DeclareMathOperator{\Aut}{Aut}

\usepackage{hyperref}
\hypersetup{
  colorlinks=true, %set true if you want colored links
   linkcolor=red,  %choose some color if you want links to stand out
}
\usepackage[hyphenbreaks]{breakurl}

\numberwithin{equation}{section}

\newcommand{\kx}{k^\times}

\begin{document}

\title[Generically free irreducible representations]{Generically free representations II:\\ irreducible representations}

\author[S. Garibaldi]{Skip Garibaldi}
\author[R.M. Guralnick]{Robert M. Guralnick}


% \date{\tt Version of \today}
%

\begin{abstract}
We determine which faithful irreducible representations $V$ of a simple linear algebraic group $G$ are generically free for $\Lie(G)$, i.e., which $V$ have an open subset consisting of vectors whose stabilizer in $\Lie(G)$ is zero.  This relies on bounds on $\dim V$ obtained in prior work (part I), which reduce the problem to a finite number of possibilities for $G$ and highest weights for $V$, but still infinitely many characteristics.  The remaining cases are handled individually, some by computer calculation.  These results were previously known for fields of characteristic zero, although new phenomena appear in prime characteristic; we provide a shorter proof that gives the result with very mild hypotheses on the characteristic.  (The few characteristics not treated here are settled in part III.)  These results
are  related to questions about invariants and the existence of a stabilizer in general position.
\end{abstract}

\maketitle

Let $G$ be a simple linear algebraic group over a field $k$ acting faithfully on a vector space $V$.  In the special case $k = \C$, there is a striking dichotomy between the properties of irreducible representations $V$ whose dimension is small (say, $\le \dim G$) versus those whose dimension is large, see \cite{AVE}, \cite{Elashvili:can}, \cite{APopov}, etc., for original results and \cite[\S8.7]{PoV} for a survey and bibliography.  For example, if $\dim V < \dim G$, then (trivially) the stabilizer $G_v$ of a vector $v \in V$ is nontrivial.  On the other hand (and nontrivially), for $\dim V$ hardly bigger than $\dim G$, the stabilizer $G_v$ for generic $v \in V$ is trivial, i.e., 1; in this case one says that $V$ is \emph{generically free} or $G$ acts \emph{generically freely} on $V$.
This property has taken on increased importance recently due to applications in Galois cohomology and essential dimension, see \cite{Rei:ICM} and \cite{M:ed} for the theory and \cite{BRV}, \cite{GG:spin}, \cite{Karp:ICM}, \cite{LMMR13}, \cite{Loetscher:fiber}, etc.~for specific applications.

With applications in mind, it is desirable to extend the results on generically free representations to all fields.  In that setting, \cite{GurLawther} has shown that, if $V$ is irreducible and $\dim V$ is large enough, then $G(k)$, the group of $k$-points of $G$, acts generically freely.  Equivalently (when $k$ is algebraically closed), the stabilizer $G_v$ of a generic $v \in V$ is an infinitesimal group scheme.  For applications, one would like to say that $G_v$ is not just infinitesimal but is the \emph{trivial} group scheme, for which one needs to know that the Lie algebra $\g$ of $G$ acts generically freely on $V$, i.e., $\g_v = 0$.  The two conditions are related in that $\dim G_v \le \dim \g_v$, so in particular if $\g_v = 0$, then $G_v$ is finite.  Conversely, if $G_v(k) = 1$, $\g_v$ can be nontrivial (i.e., $G_v$ may be a nontrivial infinitesimal group scheme), see Example \ref{C.0100.3}.

In a previous paper, \cite{GG:large}, we proved that, roughly speaking, if $\dim V$ is large enough (where the bound grows like $(\rank G)^2$), then $V$ is a generically free $\g$-module.  In this paper, we restrict our focus to irreducible modules and settle the question of whether or not $\g$ acts generically freely when $\car k$ is not special.

\begin{specthm} \label{MT}
Let $G$ be a simple linear algebraic group over a field $k$ and let $\rho \!: G \to \GL(V)$ be an irreducible representation of $G$ on which $G$ acts faithfully.   Then $V$ is generically free for $\g$ if and only if $\dim V > \dim G$ and $(G, V, \car k)$ is not in Table \ref{irred.nvfree}.
\end{specthm}

\begin{table}[htbp]
\begin{tabular}{cccrc||cccrc}
$G$&$\car k$&rep'n&$\dim V$&$\dim \g_v$&$G$&$\car k$&high weight&$\dim V$&$\dim \g_v$\\ \hline\hline
$\SL_8/\mu_4$&2&$\wedge^4$&70&3&$\Sp_8$&3&0100&40&2 \\
$\SL_9/\mu_3$&3&$\wedge^3$&84&2&$\Sp_4$&5&11&12&1\\
$\Spin_{16}/\mu_2$&2&half-spin&128&4&$\SL_4$&$p$ odd&$01p^e$, $e \ge 1$&24&1 \\
&&&&&$\SL_4/\mu_2$&2&$012^e$, $e \ge 2$&24&1
\end{tabular}
\caption{Irreducible and faithful representations $V$ with restricted highest weight of simple $G$ with $\dim V > \dim G$ that are not generically free for $\g$, up to graph automorphism.  For each, the stabilizer $\g_v$ of a generic $v \in V$ is a toral subalgebra.  The weights on the right side are numbered as in Table \ref{luebeck.table}.} \label{irred.nvfree}
\end{table}

We say that $G$ acts \emph{faithfully} on $V$ if $\ker \rho$ is the trivial group scheme.  Regardless, there is an induced map $G / \ker \rho \to \GL(V)$ that is a faithful representation of $G / \ker \rho$.

We say that $\car k$ is \emph{special} for $G$ if $\car k = p \ne 0$ and the Dynkin diagram of $G$ has a $p$-valent bond, i.e., if $\car k = 2$ and $G$ has type $B_n$ or $C_n$ for $n \ge 2$ or type $F_4$, or if $\car k = 3$ and $G$ has type $G_2$.  Equivalently, these are the cases where $G$ has a very special isogeny.  This definition of special is as in \cite[\S10]{St:rep}; in an alternative history, these primes might have been called ``extremely bad'' because they are a subset of the very bad primes --- the lone difference is that for $G$ of type $G_2$, the prime 2 is very bad but not special.  In this paper, we prove Theorem \ref{MT} when $\car k$ is \underline{not} special, and we typically assume that $\car k$ is not special in the rest of this paper.  The case where $\car k$ is special has a different flavor and will be handled in a separate paper, part III \cite{GG:special}.

%We say that $\g$ acts \emph{virtually freely} on $V$ if the stabilizer $\g_v$ of a generic vector $v \in V$ satisfies $\g_v = \ker \drho$, i.e., $\g_v$ is as small as possible.  This notion is the natural generalization of ``generically freely'' to allow for the possibility that $\g$ does not act faithfully.  (For $V$ as in Theorem \ref{MT} and not the trivial representation, 
%$\ker \drho$ is a proper ideal in $\g$, so is contained in $\z(\g)$ by the hypothesis that $\car k$ is not special for $G$.  Thus $\ker \drho = \ker \la\vert_{\z(\g)}$ for $\la$ the highest weight of $V$, which depends only on the equivalence class of $\la$ modulo the root lattice.\todo{This parenthetical remark is assuming that $\g$ does not act trivially.})

We remark that the exceptions in Theorem \ref{MT}, listed in Table \ref{irred.nvfree}, can be divided into types.  In the left column are three ``$\theta$-group'' representations, which arise from embedding $\g$ in some larger Lie algebra with a finite grading, and the generic stabilizer $G_v$ is a non-\'etale, non-infinitesimal finite group scheme.  (Premet's appendix in \cite{GG:spin} gives a detailed study of the half-spin representation of $D_8$ in characteristic 2.  For the other two representations in the left column, see Remark \ref{vinberg.equi}, \cite[4.8.2, 4.9.2]{Auld}, or \cite[4.1]{GurLawther}.) In the right column are two representations where the generic stabilizer is a nonzero infinitesimal group scheme, see Examples \ref{C.0100.3} and \ref{B.11}, and two that decompose as tensor products from Example \ref{A3.eg}.

\smallskip

The proof of Theorem \ref{MT} relies heavily on the results of part I, \cite{GG:large}, which include the case of type $A_1$ (ibid., Examples 1.8 and 10.3) and the case where $\dim V$ is large (ibid., Th.~A).  We also use Magma \cite{Magma} to verify that certain specific representations in specific characteristics are generically free, a process described in Section \ref{construct}; the key point is that it suffices to find any vector in the representations whose stabilizer is zero.  In Section \ref{converse.construct}, we prove a criterion by which Magma can verify that a representation is \emph{not} generically free for $\g$.  To handle representations where we do not specify the characteristic of the base field, we recall a means to transfer results from characteristic zero (Section \ref{transfer.sec}).  Three sections treating specific classes of representations (Sections \ref{squares.sec}--\ref{wedge.sec}) lead to the proof of Theorem \ref{MT} when $\car k$ is not special.  This proof occupies Sections \ref{restricted.sec}--\ref{final.sec}.  The first of these treats the case where the highest weight is restricted.  The second handles, roughly speaking, the case of a tensor decomposable representation.  The third and final section treats the few remaining cases, which are tensor decomposable as representations of the simply connected cover but not necessarily for $G$ itself.

\subsection*{Notation}
For convenience of exposition, we will assume in most of the rest of the paper that $k$ is algebraically closed of characteristic $p \ne 0$.  This is only for convenience, as our results for $p$ prime immediately imply the corresponding results for characteristic zero: simply lift the representation from characteristic 0 to $\Z$ and reduce modulo a sufficiently large prime.

Let $G$ be an affine group scheme of finite type over $k$.  If $G$ is additionally smooth, then we say that $G$ is an \emph{algebraic group}.  An algebraic group $G$ is \emph{simple} if it is not abelian, is connected, and has no connected normal subgroups $\ne 1, G$; for example $\SL_n$ is simple for every $n \ge 2$.

If $G$ acts on a variety $X$, the stabilizer $G_x$ of an element $x \in X(k)$ is a sub-group-scheme of $G$ with $R$-points
\[
G_x(R) = \{ g \in G(R) \mid gx = x \}
\]
for every $k$-algebra $R$.  A statement ``for generic $x$'' means that there is a dense open subset $U$ of $X$ such that the property holds for all $x \in U$.

If $\Lie(G) = 0$ then $G$ is finite and \'etale.  If additionally $G(k) = 1$, then $G$ is the trivial group scheme $\Spec k$.

We write $\g$ for $\Lie(G)$ and similarly $\spin_n$ for $\Lie(\Spin_n)$, etc.  We put $\z(\g)$ for the center of $\g$; it is the Lie algebra of the (scheme-theoretic) center of $G$.  As $\car k = p$, the Frobenius automorphism of $k$ induces a ``$p$-mapping'' $x \mapsto x^{[p]}$ on $\g$.
When $G$ is a sub-group-scheme of $\GL_n$ and $x \in \g$, the element $x^{[p]}$ is the $p$-th power of $x$ with respect to the typical, associative multiplication for $n$-by-$n$ matrices, see \cite[\S{II.7}, p.~274]{DG}.   An element $x \in \g$ is \emph{nilpotent} if $x^{[p]^n} = 0$ for some $n > 0$, \emph{toral} if $x^{[p]} = x$, and \emph{semisimple} if $x$ is contained in the Lie $p$-subalgebra of $\g$ generated by $x^{[p]}$, cf.~\cite[\S2.3]{StradeF}.

  \subsection*{Acknowledgements} We thank the  referees for their detailed and helpful comments on earlier versions of this paper.
Guralnick was partially supported by NSF grants DMS-1265297, DMS-1302886, and DMS-1600056.

%%%%%%%%%%%%%%%%%%%%%%%%%%%%%  
\section{Dominant weights and faithful representations}

Let $G$ be a reductive group over the (algebraically closed) field $k$, and fix a pinning for $G$ which includes a maximal torus $T$, a Borel subgroup containing $T$, and generators for the root subalgebras of $\g$.
 
Irreducible representations of $G$, up to equivalence, are in one-to-one correspondence with dominant weights (relative to the fixed pinning), where the correspondence is given by sending a representation to its highest weight.  We write $L(\la)$ for the irreducible representation with highest weight $\la$, imitating the notation in \cite{Jantzen}.

We number dominant weights as in Table \ref{luebeck.table}, imitating \cite{luebeck} to make it convenient to refer to that paper.  We write $c_1 c_2 c_3 \cdots c_\ell$ as shorthand for the dominant weight $\sum c_i \omega_i$, where $\omega_i$ is the fundamental dominant weight corresponding to the vertex $i$ in Table \ref{luebeck.table}.  The weight $\sum c_i \omega_i$ is \emph{restricted} if $\car k = 0$, or if $\car k \ne 0$ and $0 \le c_i < \car k$ for all $i$.

\begin{table}[bth]
{\centering\noindent\makebox[450pt]{
\begin{tabular}[c]{p{2.2in}|p{3in}}

${(A_\ell)~~}$
\begin{picture}(7,2)(0,0)
\put(0,1){\circle*{3}}
\put(0,1){\line(1,0){20}}
\put(20,1){\circle*{3}}
\put(20,1){\line(1,0){20}}
\put(40,1){\circle*{3}}
\put(40,-1.6){ \mbox{$\cdots$}}
\put(62,1){\circle*{3}}
\put(62,1){\line(1,0){20}}
\put(82,1){\circle*{3}}
\put(82,1){\line(1,0){20}}
\put(102,1){\circle*{3}}

\put(-2,-7){\mbox{\tiny $1$}}
\put(18,-7){\mbox{\tiny $2$}}
\put(38,-7){\mbox{\tiny $3$}}
\put(54,-7){\mbox{\tiny $\ell$$-$$2$}}
\put(75,-7){\mbox{\tiny $\ell$$-$$1$}}
\put(100,-7){\mbox{\tiny $\ell$}}
\end{picture}
\vspace{0.5cm}

&

${(C_\ell)~~}$
\begin{picture}(7,2)(0,0)
\put(0,1){\circle*{3}}
\put(0,0){\line(1,0){20}}
\put(0,2){\line(1,0){20}}
\put(20,1){\circle*{3}}
\put(20,1){\line(1,0){20}}
\put(40,1){\circle*{3}}
\put(40,-1.6){ \mbox{$\cdots$}}
\put(62,1){\circle*{3}}
\put(62,1){\line(1,0){20}}
\put(82,1){\circle*{3}}
\put(82,1){\line(1,0){20}}
\put(9,-1){{\small\mbox{$>$}}}
\put(102,1){\circle*{3}}

\put(-2,-7){\mbox{\tiny $1$}}
\put(18,-7){\mbox{\tiny $2$}}
\put(38,-7){\mbox{\tiny $3$}}
\put(54,-7){\mbox{\tiny $\ell$$-$$2$}}
\put(75,-7){\mbox{\tiny $\ell$$-$$1$}}
\put(100,-7){\mbox{\tiny $\ell$}}
\end{picture}
\vspace{0.5cm}
\\

${(B_\ell)~~}$
\begin{picture}(7,2)(0,0)
\put(0,1){\circle*{3}}
\put(0,0){\line(1,0){20}}
\put(0,2){\line(1,0){20}}
\put(20,1){\circle*{3}}
\put(20,1){\line(1,0){20}}
\put(40,1){\circle*{3}}
\put(40,-1.6){ \mbox{$\cdots$}}
\put(62,1){\circle*{3}}
\put(62,1){\line(1,0){20}}
\put(82,1){\circle*{3}}
\put(82,1){\line(1,0){20}}
\put(9,-1){{\small\mbox{$<$}}}
\put(102,1){\circle*{3}}

\put(-2,-7){\mbox{\tiny $1$}}
\put(18,-7){\mbox{\tiny $2$}}
\put(38,-7){\mbox{\tiny $3$}}
\put(54,-7){\mbox{\tiny $\ell$$-$$2$}}
\put(75,-7){\mbox{\tiny $\ell$$-$$1$}}
\put(100,-7){\mbox{\tiny $\ell$}}
\end{picture}
\vspace{0.5cm}

&

${(D_\ell)~~}$
\begin{picture}(7,2)(0,0)
\put(102,1){\circle*{3}}
\put(82,1){\line(1,0){20}}
\put(20,1){\circle*{3}}
\put(20,1){\line(1,0){20}}
\put(40,1){\circle*{3}}
\put(40,-1.6){ \mbox{$\cdots$}}
\put(62,1){\circle*{3}}
\put(62,1){\line(1,0){20}}
\put(82,1){\circle*{3}}
\put(20,2){\line(-4,3){15}}
\put(20,0){\line(-4,-3){15}}
\put(5,12.9){\circle*{3}}
\put(5,-10.9){\circle*{3}}

\put(-2,-12){\mbox{\tiny $2$}}
\put(-2,11.8){\mbox{\tiny $1$}}
\put(18,-7){\mbox{\tiny $3$}}
\put(38,-7){\mbox{\tiny $4$}}
\put(54,-7){\mbox{\tiny $\ell$$-$$2$}}
\put(75,-7){\mbox{\tiny $\ell$$-$$1$}}
\put(100,-7){\mbox{\tiny $\ell$}}
\end{picture}

\end{tabular}
}}
\caption{Dynkin diagrams of simple root systems of classical type, with simple roots numbered as in \cite{luebeck}.} \label{luebeck.table}
\end{table}

The paper \cite{luebeck} studies $L(\la)$ when $\la$ is restricted.  When $\la$ is not restricted, we have the following statement: \emph{If $\la = \la_0 + p\la_1$ for $\la_0 \in T^*$ dominant and restricted, $\la_1 \in T^*$ dominant, and $p = \car k$, then $L(\la) \cong L(\la_0) \otimes L(\la_1)^{[p]}$}, see \cite[II.3.16]{Jantzen}.  Here $L(\la_1)^{[p]}$ denotes the Frobenius twist of the representation $L(\la_1)$ as in \cite[I.9.10]{Jantzen}; $\g$ acts trivially on it.

Slightly more delicate analysis is required to handle the case where $\la_0, \la_1$ are not assumed to belong to $T^*$, see Example \ref{A3.eg} for an illustration.

\begin{lem} \label{faith}
Let $G$ be a simple algebraic group over $k$ of characteristic $p \ne 0$.  Let $\la = \la_0 + p\la_1$ for $\la_0, \la_1$ dominant weights and $\la_0$ restricted.
\begin{enumerate}
\item \label{faith.0} If $G$ acts faithfully on $L(\la)$, then $\la_0 \ne 0$.
\item \label{faith.1} If $\la_0$ and $\la_1$ are both nonzero, then $\dim L(\la) > \dim G$.
\end{enumerate}
\end{lem}

\begin{proof}
We write $\Gt$ for the simply connected cover of $G$.

For \eqref{faith.0}, suppose $\la_0 = 0$.  Then $L(\la)$ is isomorphic to a Frobenius twist $L(\la_1)^{[p]}$ as a representation of $\Gt$, and the composition $\gt \to \g \to \gl(L(\la))$ is zero.  As $\gt \to \g$ is not itself the zero map, $\g$ does not act faithfully on $L(\la)$.

For \eqref{faith.1}, 
we may assume that $G$ is simply connected, so that $L(\la) \cong L(\la_0) \otimes L(\la_1)^{[p]}$ and $\dim L(\la) = \dim L(\la_0) \cdot \dim L(\la_1)$.  This, in turn, is at least $m(G)^2$ for $m(G)$ the dimension of the smallest nontrivial irreducible representation of $G$.  We list these values in Table \ref{mG.table}, obtained from \cite{luebeck}.  We note that in each case
\begin{equation} \label{mG.ineq}
m(G)^2 > \dim G,
\end{equation}
proving \eqref{faith.1}.
\end{proof}

\begin{table}[hbt]
\[
\begin{array}{c|ccccccccc}
\text{type of $G$}&A_\ell&B_\ell\, (\ell \ge 3)&C_\ell\, (\ell \ge 2)&D_\ell \, (\ell \ge 4)&E_6&E_7&E_8&F_4&G_2 \\ \hline
m(G)&\ell+1&2\ell+1-\eps&2\ell&2\ell&27&56&248&26-\eps&7-\eps \\ 
\dim G&\ell^2+2\ell&2\ell^2+\ell&2\ell^2+\ell&2\ell^2-\ell&78&133&248&56&14 
\end{array}
\]
\caption{The dimension $m(G)$ of the smallest nontrivial irreducible representation of $G$, assuming $\car k$ is not special for $G$.  The symbol $\eps$ represents 0 or 1 depending on $\car k$, and is 0 except possibly when $\car k \in \{ 2, 3 \}$.} \label{mG.table}
\end{table}

%%%%%%%%%%%%%%%%%%%%%%%%%%%%%%%%%%%%%%%%%%%%%
\section{Results from part I} \label{results}

For a representation $V$ of simple $G$ with $V^{[\g,\g]} = 0$ (such as in Theorem \ref{MT} in this paper), we showed in part I (\cite[Th.~A]{GG:large}): \emph{if $\dim V > b(G)$ for $b(G)$ as in Table \ref{classical.table} and $\car k$ is not special, then $\g$ acts virtually freely on $V$}.   Here, virtually free means that the stabilizer $\g_v$ for a generic $v \in V$ equals $\ker [ \g \to \gl(V) ]$, i.e., $\g_v$ is as small as possible.  It is the natural notion that generalizes ``generically free'' to allow for the case where the kernel is not zero.

\begin{table}[thb]
\begin{tabular}{cc|c||cc|r} 
type of $G$&$\car k$&$b(G)$&type of $G$&$\car k$&$b(G)$ \\ \hline
$A_\ell$&$\ne 2$& $2.25(\ell+1)^2$&$G_2$&$\ne 3$&48 \\
$A_\ell$&$= 2$&$2\ell^2 + 4\ell$&$F_4$&$\ne 2$&240 \\
$B_\ell$&$\ne 2$&$8\ell^2$&$E_6$&any&360 \\
$C_\ell$&$\ne 2$&$6\ell^2$&$E_7$&any&630\\
$D_\ell$&$\ne 2$&$2(2\ell-1)^2$&$E_8$&any&1200 \\
$D_\ell$&$=2$&$4\ell^2$
\end{tabular}
\caption{Bound $b(G)$ from part I} \label{classical.table}
\end{table}

Recall that $\lie := \Lie(G)$.  For $x \in \lie$, put
\[
V^x := \{ v \in V \mid \drho(x)v = 0 \}
\]
and $x^G$ for the $G$-conjugacy class $\Ad(G)x$ of $x$.
We are going to verify the inequality
\begin{equation} \label{ineq.mother}
  \dim x^G + \dim V^x < \dim V
\end{equation}
for various $x \in \g$.
The following lemma is Lemma 2.6 in \cite{GG:spin}; it resembles \cite[Lemma 4]{AndreevPopov} and \cite[\S3.3]{Guerreiro}.
\begin{lem} \label{ineq}
Suppose $G$ is semisimple over an algebraically closed field $k$ of characteristic $p > 0$, and  let $\lsub$ be a $G$-invariant subspace of $\lie$.  
\begin{enumerate}
\item \label{mother.lie2} If inequality \eqref{ineq.mother} holds for every toral or nilpotent $x \in \g \setminus \lsub$, then  $\g_v \subseteq \lsub$ for generic $v \in V$.

\item \label{mother.lie} If $\lsub$ consists of semisimple elements and \eqref{ineq.mother} holds for every $x \in \g \setminus \lsub$ with $x^{[p]} \in \{ 0, x \}$, then $\g_v \subseteq \lsub$  for generic $v$ in $V$. $\hfill\qed$
\end{enumerate}
\end{lem}

Taking $\lsub = \z(\g)$ in Lemma \ref{ineq}, we see that verifying \eqref{ineq.mother} for nonzero nilpotent and noncentral toral $x \in \g$ implies that $\g_v \subseteq \z(\g)$ for generic $v \in V$.   This in turn implies that the action is virtually free since $Z(G)$ is a diagonalizable group scheme for the $G$'s we consider here (so $\g_v \cap \z(\g) = \z(\g)_v = \ker \drho$).

Theorem 12.2 in \cite{GG:large} proved a somewhat stronger result than the one stated at the start of this section: \emph{If $V$ is a representation of a simple group $G$ such that $V^{[\g, \g]} = 0$ and $\dim V > b(G)$, then \eqref{ineq.mother} holds for all noncentral $x \in \g$ with $x^{[p]} \in \{ 0, x \}$.} 

%%%%%%%%%%%%%%%%%%%%%%%%%%%%%%%%%%%%%%%%%%%%%%%%%%%%%%%%%%%%%%%%%%%%%
\section{Constructing representations in Magma} \label{construct}

In order to prove Theorem \ref{MT}, the results of \cite{GG:large} reduce us to considering a finite list of irreducible representations, each of which we will consider.  Some of these will be dealt with by invoking calculations done on a personal computer using Magma \cite{Magma}, which we now explain.  (Code and output are attached to the arxiv version of this article.)  

The Magma instructions 
\begin{align*}
&\texttt{R := IrreducibleRootDatum(}T, \ell\texttt{);}\\
&\g \texttt{ := LieAlgebra(R, GF(}q\texttt{));}
\end{align*}
create a Lie algebra $\g$ of the split reductive group over the finite field $\F_q$ with root datum \texttt{R} of type $T_\ell$ (by default simply connected).  For a given highest weight $\la$, 
\begin{align*}
&\texttt{HighestWeightRepresentation(}\g, \la\texttt{);}
\end{align*}
 gives a homomorphism $\rho$ from $\g$ to matrices, and one can identify the space of row vectors $v$ where the action by $\g$ is $v \mapsto v\rho(x)$ for $x \in \g$ with the representation $H^0(\la)$ of $G$ in the notation of \cite{Jantzen}; it is induced from the 1-dimensional representation $\la$ of the Borel subgroup.
 The vector $(1, 0, \ldots, 0)$, the first basis vector in the Magma ordering, is a highest weight vector and it generates a submodule $V$ that is irreducible with highest weight $\la$.  (Untrusting readers can verify that the submodule generated by this vector has the same dimension as the irreducible representation with the same highest weight as recorded in the literature, and therefore the submodule is the desired irreducible representation,)
 
For any row vector $v$, it is then a matter of linear algebra to compute the stabilizer $\g_v$, i.e., the subspace of $x \in \g$ such that $v\rho(x) = 0$.  It is determined by \texttt{Kernel(VerticalJoin([}$v\rho(y)$ \texttt{:} $y$\texttt{ in Basis($\g$)]))}.

To verify that a particular $V$ is virtually free, we use \texttt{Random(V)} to generate random vectors $v \in V$.  For each, we compute $\dim \g_v$.  By upper semicontinuity of dimension, $\dim \g_v$ is at least as big as $\dim \g_w$ for $w$ generic in $V$.  Therefore, if we find any $v \in V$ with $\dim \g_v = \dim \ker \drho$, we have verified that the representation is virtually free.

\begin{rmk*}
Suppose $q \!: \Gt \to G$ is a central isogeny; note that the differential $\dq \!: \gt \to \g$ need not be surjective, i.e., $\ker q$ need not be \'etale.  Nonetheless, if $\g$ acts virtually freely on $V$, then so does $\gt$.  Therefore, in the computer calculations described above we work with $\g$, the Lie algebra of the group $G$ that acts faithfully on $V$.  In Magma, this can be done by invoking the optional argument \texttt{Isogeny} for \texttt{IrreducibleRootDatum}.

(If we instead assume that $\gt$ acts virtually freely on $V$, it may occur that $\g$ does not.  For example, that is the case when $\car k = 2$ and (a) $\Gt = \SL_4$, $G = \PGL_4$, and $V$ has highest weight $\omega_2 + 2\omega_3$ as in Example \ref{A3.eg} or (b) $\Gt = \Sp_8$, $G = \PSp_8$,  and $V$ is the 16-dimensional irreducible ``spin'' representation as in \cite[\S8]{GG:special}.)
\end{rmk*}

%%%%%%%%%%%%%%%%%%%%%%%%%%%%%%%%%%%%%%%%%%%%%%%%%%%%%%%%%%%%%%%%%%%%%
\section{Examples where $\g$ does not act virtually freely}  \label{converse.construct}

\begin{lem} \label{nvfree}
Let $V$ be a representation of a reductive algebraic group $G$, and suppose that Cartan subalgebras in $\g$ are maximal toral subalgebras\footnote{This condition is equivalent to condition (2) in Lemma \ref{regular.lem} by \cite[XIII.6.1d]{SGA3.2}.}.
 If there is a $v \in V$ such that 
\begin{enumerate}
	\item $\lsub := \g_v$ is a toral subalgebra;
	\item $\dim \z_\g(\lsub) = \rank G$; and
	\item \label{nvfree.3} $\dim G - \rank G = \dim V - \dim V^\lsub$,
\end{enumerate}
then there is an open subset $U$ of $V$ containing $v$ such that $\g_u$ is a $G$-conjugate of $\lsub$ for every $u \in U$ and there is a maximal torus $T$ such that $G_u$ is $G$-conjugate to a closed sub-group-scheme of $N_G(T)$ for every $u \in U$.
\end{lem}

\begin{proof} 
Because $G$ is reductive and $\lsub$ is toral, there exists a maximal torus $T$ in $G$ whose Lie algebra $\tor$ contains $\lsub$, see \cite[Th.~13.3, Rmk.~13.4]{Hum:p}.  Since $\z_\g(\lsub)$ contains $\tor$, the two are equal.   In particular, $T$ normalizes $V^\lsub$.  Moreover, any element of $G$ that normalizes $\lsub$ also normalizes $\z(\lsub) = \tor$, so $N_G(\lsub) \subseteq N_G(\tor) = N_G(T)$ (where the latter equality is by the hypothesis on $\g$ \cite[XIII.6.1b]{SGA3.2}) and $N_G(\lsub)^\circ = T$.

Put $\hat{U}$ for the set of $v' \in V^\lsub$ such that $\dim \g_{v'}$ is minimal; it is open in $V^\lsub$.  On the one hand, $\lsub \subseteq \g_{v'}$, and on the other hand, $v \in V^\lsub$, so $\dim \g_{v'} \le \dim \lsub$, whence $\g_{v'} = \lsub$ for all $v' \in \hat{U}$ and $v$ is in $\hat{U}$.  

We claim that, for generic $w \in V^\lsub$, the transporter $\Tran_G(w, V^\lsub)$, which has $R$-points
\[
\Tran_G(w, V^\lsub)(R) = \{ g \in G(R) \mid gw \in V^\lsub \ot R \}
\]
for every $k$-algebra $R$, equals $N_G(\lsub)$.  The direction $\supseteq$ is clear because $N_G(\lsub)$ normalizes $V^\lsub$.  Conversely, suppose $gw \in V^\lsub \ot R$.  Then $gw \in \hat{U}(R)$ by the definition of $\hat{U}$, so $\g_{gw} = \lsub$ and $\subseteq$ is verified.

Define $\psi \!: G \times V^\lsub \to V$ by $\psi(g,w) = gw$.  By the preceding paragraph, for generic $w \in V^\lsub$, $\psi^{-1}(w) = \{ (g, g^{-1}w) \mid g \in N_G(\lsub) \}$.  That is,
\[
\dim \im \psi = \dim G + \dim V^\lsub - \dim N_G(\lsub),
\]
which is $\dim V$ by \eqref{nvfree.3}.
Thus $\psi$ is dominant and there is an open subset $U$ of $V$ consisting of elements whose stabilizer in $\g$ is conjugate to $\lsub$.
\end{proof}

In the language of \cite[\S2.8]{PoV}, the proof shows that $V^\lsub$ is a ``relative section'' for the action of $G$ on $V$.

The hypotheses of Lemma \ref{nvfree} are easy to verify with a computer.  For example, to check that $\g_v$ is toral, one checks that it is abelian (Magma's \texttt{IsAbelian}) and that a basis consists of semisimple elements (by checking, for each basis vector $x$, that $x$ belongs to the subspace spanned by $x^{[p]^i}$ for $i \ge 1$).

\begin{eg}[$C_4$, 0100, $p = 3$] \label{C.0100.3}
Consider now $G = \Sp_8$ over a field $k$ of characteristic 3.  (See Prop.~\ref{C.wedge3}\eqref{C.sp8} for the case $\car k \ne 2, 3$.)  It has a unique irreducible representation $V$ with $\dim V = 40$ \cite{luebeck}, which occurs as a quotient of the Weyl module of dimension 48 contained in $\wedge^3(k^8)$ (with $k^8$ as the other composition factor), compare \cite{PremetSup} or Proposition \ref{C.wedge3}.  Using Magma, one can construct $V$ (say, with $k = \F_3$) as in the preceding section and verify that for a random $v \in V$, in the notation of Lemma \ref{nvfree}, $\dim \lsub = 2$ and $\dim V^\lsub = 8$.  It follows that $\g$ does not act virtually freely on $V$.  On the other hand, $G_v(k) = 1$ for generic $v \in V$ by \cite{GurLawther}, so this is an example of a representation where the scheme-theoretic generic stabilizer $G_v$ is a nontrivial and infinitesimal group scheme.
\end{eg}

Lemma \ref{nvfree} shows also that the second representation in the right column of Table \ref{irred.nvfree} is not virtually free, see Example \ref{B.11}.

%%%%%%%%%%%%%%%%%%%%%%%%%%%%%%%%%%%%%%%%%%%%%%%%%%%%%%%%%%%%%%%%%%%%%
\section{Representations defined over a localization of the integers} \label{transfer.sec}

Recall that $G$ is defined over an algebraically closed field $k$ of characteristic $p$, and in particular is split.  Let now $R$ be a subring of $\Q$ with homomophisms to $\F_p$ and to a field $K$ containing a primitive $p$-th root of unity $\zeta$ (e.g., take $R = \Z$ and $K = \C$).  There exists a smooth affine group scheme $G_R$ over $R$ which is split and such that $G_R \times k$ is isomorphic to $G$.  

\begin{lem} \label{auld.popov}
Let $\rho \!: G_R \to \GL(V)$ be a homomorphism of group schemes over $R$ for some free $R$-module $V$.  Then the following are equivalent:
\begin{enumerate}
	\item \label{auld.popov.x} $\dim x^G + \dim (V_k)^x < \dim V$ for all noncentral $x \in \g$ such that $x^{[p]} = x$.
	\item \label{auld.popov.g} $\dim g^{G_K} + \dim (V_K)^g < \dim V$ for all noncentral $g \in G_R(K)$ such that $g^p = 1$.
\end{enumerate}
\end{lem}

Here and below we use the shorthand $X_F$ for $X_R \times F$, where $X_R$ is an $R$-scheme and there is an implicit homomorphism $R \to F$.

\begin{proof}
This is essentially \S3.4 in \cite{Auld}, which we reproduce here for the convenience of the reader.  Pick a split maximal torus $T_R$ in $G_R$ and a basis $\tau_1, \ldots, \tau_\ell$ of the lattice of cocharacters $\Gm \to T_R$.  Identifying the Lie algebra of $\Gm$ with $k$, the elements $h_j := \dtau_j(1)$ make up a basis of the Lie algebra $\tor$ of $T_R \times \F_p$ such that $h_j^{[p]} = h_j$.  This gives a bijection of toral elements in $\tor$ with elements of order $p$ in $T_K$ via
\[
\psi \!: \sum c_j h_j \mapsto \prod \tau_j(\zeta^{c_j}) \quad \text{for $c_j \in \F_p$.}
\]

There is a basis $\chi_1, \ldots, \chi_\ell$ of the lattice of characters $T_R \to \Gm$ such that $\chi_i \circ \tau_j \!: \Gm \to \Gm$ is the identity for $i = j$ and trivial for $i \ne j$, hence $\dchi_i(h_j) = \delta_{ij}$ for all $i, j$.  Writing a character $\chi$ as $\sum d_i \chi_i$ for $d_i \in \Z$, we find 
\[
\chi(\psi(\textstyle\sum c_i h_i)) = \prod_i \zeta^{c_i d_i} = \zeta^{\sum c_i d_i} = \zeta^{\dchi(\sum c_i h_i)}
\]
for $c_i \in \F_p$.  That is, for toral $x \in \tor$, $\dchi(x) = 0$ in $\F_p$ if and only if $\chi(\psi(x)) = 1$.  Decomposing $V$ as a sum of weight spaces relative to $T_R$ (using that $R$ is an integral domain), we find that $\dim (V_k)^x = \dim (V_K)^{\psi(x)}$.

The centralizer in $\g$ of $x$ and the centralizer in $G_K$ of $\psi(x)$ contain $\Lie(T_k)$ and $T_K$, so their identity components are generated by that and the root subalgebras or subgroups corresponding to roots vanishing on $x$ or $\psi(x)$ respectively.  As in the preceding paragraph, we find that
the centralizers of $x$ and $\psi(x)$ have the same dimension, hence (a) $x$ is central in $\g$ if and only if $\psi(x)$ is central in $G_K$ and (b) $\dim x^G = \dim \psi(x)^{G_K}$.  The equivalence of \eqref{auld.popov.x} and \eqref{auld.popov.g} follows.  
\end{proof}

We now consider five examples and show, in most cases, that inequality  \eqref{ineq.mother} holds.  We use Lemma \ref{auld.popov} to handle the
elements with $x^{[p]}=x$.   In the cases where the characteristic $p$ module is the reduction of a characteristic $0$ module, it suffices
to prove the inequality for elements of order $p$ in the group over $\mathbb{C}$.   In all the examples below, this has been confirmed in
\cite[2.5.10, 2.5.17, 2.5.18, 2.5.24, 2.6.10]{GurLawther}.   It is also straightforward to use Magma to compute this in all the examples below as the modules  have small dimension.   One
can also use closure arguments to reduce to the case of nilpotent elements.  
Thus, it suffices to consider elements $x$ with $x^{[p]}=0$. 

\begin{eg}[$B_2$, 11] \label{B.11}
Let $G = \Spin_5 \cong \Sp_4$ and take $V$ to be the irreducible representation of dimension 12 (if $\car k = 5$) or 16 
(if $\car k \ne 5$).  It occurs as a composition factor of the tensor product $X$ of the two fundamental irreducible representations.

In case $\car k = 5$, we apply Lemma \ref{nvfree}.  One finds $\dim L = 1$ and $\dim V^L = 4$, so $V$ is not virtually free.  
We remark that in this case again $G_v(k) = 1$, so $G_v$ is a nonzero infinitesimal group scheme.

In case $\car k = 2$, we verify that $V$ is generically free for $\g$ using Magma as in \S\ref{construct}.

So assume $\car k \ne 2, 5$.  As $X$ is self-dual, it is a direct sum of $V$ and $X/V$, the natural representation of $\Sp_4$.  
In this case we argue that $V$ is virtually free by verifying \eqref{ineq.mother}.  

A long root element $x$ has a single Jordan block of size 2 on the natural module and 2 Jordan blocks of size 2 on the 5-dimensional module.  Since $\car k \ne 2$, $x$ has partition $(3^2, 2^5, 1^4)$ on $X$, so $\dim X^x = 11$.  Since $X/V$ is the 4-dimensional symplectic module, $\dim (X/V)^x = 3$, so $\dim V^x = 8$.  As $\dim x^G = 4$, \eqref{ineq.mother} is verified.

For any other nilpotent class, the closure of $x^G$ in $\sp_4$ contains a nilpotent element
with partition $(2,2)$, so $\dim V^x \le 6$; as $\dim x^G \le 8$, the inequality is verified.
\end{eg}

\begin{eg}[$B_3$, 101] \label{B.101}
Let $G = \Spin_7$ and take $V$ to be the irreducible representation of 
dimension 40 (if $\car k = 7$) or 48 (if $\car k \ne 7$).  It occurs as a composition factor of the tensor product $X$ of the natural and spin representations.
In case $\car k = 2$ or $7$, we construct the representation in Magma as in \S\ref{construct} and observe that it is generically free.  

So suppose $\car k \ne 2, 7$.  Then $X$ is self-dual so it is a direct sum of $V$ and $X/V$, the spin representation.  As in the preceding example, we argue that $V$ is virtually free by verifying \eqref{ineq.mother}.  Suppose that $x$ is nilpotent.   For $x$ with partition $(3^2, 1)$ on the natural representation, $\dim V^x \le 22$ and $\dim x^G = 14$.  
If $x$ has partition $(7)$ or $(5, 1^2)$, then $\dim x^G \le 18$ and $\dim V^x \le 22$ (by specialization).
A long root element $x$ (partition $(2^2, 1^3)$), has $\dim x^G = 8$ and $\dim V^x = 34$.  
The remaining possibilities for $x$ have partition $(3, 2^2)$ or $(3,1^4)$, which have $\dim x^G = 12$ or $10$ and by specialization $\dim V^x \le 34$.
\end{eg}

\begin{eg}[$D_4$, 1001] \label{D.1001}
Consider the representation $V$ of $G = \Spin_8/\mu_2$ with highest weight 1001.  In case $\car k = 2$, $\dim V = 48$ and 
we verify with Magma that $V$ is generically free for $\g$.  

So suppose $\car k \ne 2$, in which case $\dim V = 56$.  Writing $V_i$ with $i = 1, 2, 3$ for the three inequivalent irreducible 8-dimensional representations, 
we find $X := V_1 \ot V_2 \cong V \oplus V_3$.

Suppose that $x$ is nonzero nilpotent with $\dim x^G < 22$.  Certainly $\dim V^x \le \dim V^y$ for a root element $y$.  Such a $y$ has two Jordan blocks of 
size 2 on the $V_i$'s, and so $y$ acts on $X$ with partition $(3^4, 2^{16}, 1^{20})$.
Thus $\dim X^y = 40$ and $\dim V^y = \dim X^y - \dim V_3^y = 34$, and the inequality is verified for $x$.

We now divide into cases based on the partition of $x$ on one of the $V_i$'s.
If $x$ only has Jordan blocks of size at most 3, then $\dim x^G < 21$ and we are done by the previous paragraph. 

 If $x$ has two Jordan blocks of size 4, then $\dim V^x < 16$.  If $x$ has a Jordan block of size $\ge 5$, then $\dim V^x < 20$.  
 In either case, as $\dim x^G \le 24$, the inequality is verified.  In summary, $V$ is generically free for $\g$.
\end{eg}

\begin{eg}[$D_5$, 20000, $\car k \ne 2$] \label{D.20000}
Consider the representation $V$ of $G = \SO_{10}$ with highest weight 20000 of dimension 126 over a field $k$ of characteristic different from 2.
 For one of the half-spin representations $X$, the second symmetric power $\Sym^2 X$ is a direct sum of $V$ and the natural 10-dimensional module.

A root element $x \in \g$ has a 12-dimensional fixed space on $X$ and so has 4 nontrivial Jordan blocks.  On $\Sym^2 X$, it has a fixed space of 
dimension 84 hence $\dim V^x = 76$. 
Therefore, for every nonzero nilpotent $x \in \g$, we have $\dim V^x \le 76$ and of course $\dim x^G \le \dim G - \rank G = 40$, verifying the inequality, so $V$ is generically free for $\g$.
\end{eg}

\begin{eg}[$C_5$, 10000, $\car k \ne 2$] \label{C.10000}
Let now $V$ be the irreducible representation of $G = \Sp_{10}$ with highest weight 10000.  
In case $\car k = 3$, $\dim V = 122$ and one checks using Magma that a generic vector has trivial stabilizer.  
So assume $\car k > 3$, in which case $\dim V = 132$.

As above it is enough to 
verify the inequality for nilpotent elements of $\sp_{10}$.  Restricting to the Levi subgroup $\Sp_8$, the representation decomposes as a direct sum of irreducibles $X \oplus Y \oplus Y$ where $\dim X = 48$ and $\dim Y = 42$.  Since $\car k  > 3$, $X$ is a submodule of $\wedge^3 k^8$ with quotient $k^8$.  The restriction of $Y$ to the Levi $\Sp_6$ in $\Sp_8$ is a direct sum of irreducibles $Y' \oplus Y' \oplus Y''$ where $\dim Y' = \dim Y'' = 14$, $Y'$ is a submodule of $\wedge^3 k^6$ with quotient $k^6$ and $Y''$ is a submodule of $\wedge^2 k^6$ with quotient $k$.  Using these decompositions, we find that a long root $x \in \sp_6 \subset \sp_{10}$ has $\dim V^x = 90$ and nilpotent $y \in \sp_6 \subset \sp_{10}$ with partition $(4, 1^6)$ has $\dim V^y = 19$.  In view of the fomer, it suffices to consider nilpotent $z \in \g$ such that $\dim C_{\Sp_{10}}(z) \le 13$  Such a $z$ has a Jordan block of size at least 4 and so specializes to $y$.  Then $\dim z^{\Sp_{10}} + \dim V^z \le 50 + 19$, verifying the inequality. 
\end{eg}

%%%%%%%%%%%%%%%%%%%%%%%%%%%%%%%%%%%%%%%%%%%%%%%%%%
\section{Example: symmetric squares and wedge squares} \label{squares.sec}

Recall that $k$ is assumed algebraically closed of characteristic $p \ge 0$.  Put $\gl_n$ for the Lie algebra of $n$-by-$n$ matrices with entries in $k$.
We first note that, for $x \in \gl_n$, $Z_{\GL_n}(x)$ is the group of units in the associative $k$-algebra $\z_{\gl_n}(x)$.  Therefore, $\dim x^{\GL_n} = \dim\, [\gl_n, x]$ and we have the following well-known result.

\begin{lemma} \label{ineq.GL}
Let $x \in \gl_n$.   Then $\dim x^{\GL_n} + \dim \z_{\gl_n}(x) = n^2$.$\hfill\qed$
\end{lemma}

Suppose that $x \in \gl_n=\gl(V)$  is nilpotent.  It is well known that $\dim x^{\GL_n}$, and therefore also $\dim \z_{\gl_n}(x)$, depends only on the Jordan form of $x$ and not on $k$.

\begin{lemma} \label{independent}  Let $x \in \gl_n = \gl(V)$ be nilpotent and assume that $p \ne 2$.
Then $\dim (\Sym^2V)^x$ and $\dim (\wedge^2V)^x$ are independent of the characteristic.
In particular if $x \in \so_n$, then $\dim x^{\SO_n} + \dim (\wedge^2V)^x = \dim \so_n$.
\end{lemma}

\begin{proof}  
Since $\car k \ne 2$, as $x$-modules $\gl_n$ and $V \otimes V$ are isomorphic (since $V \cong V^*$ for $x$).   Since
$V \otimes V \cong \Sym^2 V \oplus \wedge^2V$ and since the dimension of the fixed
space of $x$ can only increase when reducing modulo a prime ($x$ acting on $V$
is defined over the integers).  The first claim follows.

For the second, $\wedge^2V$ is the adjoint module for $\SO_n$, so the equality holds in characteristic
$0$.  Since $\dim (\wedge^2 V)^x$ depends only on the Jordan form of $x$ and not on $k$, and it is well known that $\dim x^{\SO_n}$ also does not (as $p \ne 2$), the equality also holds over $k$.
\end{proof}

\begin{lemma}  \label{GL.reg}
Let $x \in \gl_n =\gl(V)$ with $x$ a regular nilpotent element.
\begin{enumerate}
\item \label{GL.ad} The number of Jordan blocks of $x$ on $\gl(V)$ and $V\otimes V$ is $n$.
\end{enumerate}
If furthermore $\car k \ne 2$, then
\begin{enumerate}
\setcounter{enumi}{1}
\item \label{GL.sym}
the number of Jordan blocks
of $x$ on $\Sym^2V$ is $n/2$ if $n$ is even and $(n+1)/2$ if $n$ is odd; and
\item \label{GL.wedge}
the number of Jordan blocks
of $x$ on $\wedge^2V$ is $n/2$ if $n$ is even and $(n-1)/2$ if $n$ is odd.
\end{enumerate}
\end{lemma}

\begin{proof} 
As $x$ is nilpotent, $V$ and $V^*$ are equivalent $k[x]$-modules, hence the number of Jordan blocks on $V \ot V$ and $\gl(V)$ is the same and is also independent of the characteristic.
By Lemma \ref{independent}, we may assume that $k$ has characteristic $0$. 

In characteristic $0$, we view
$V$ as a module under a principal $\SL_2$ and see that $V \otimes V \cong L(n-1) \ot L(n-1) \cong L(2n-2) \oplus L(2n-4) \oplus \cdots \oplus L(0)$,
proving \eqref{GL.ad}.  Examining the weights shows that $\wedge^2 V \cong L(2n-4) \oplus L(2n-8) \oplus \cdots$, proving \eqref{GL.wedge}, from which \eqref{GL.sym} follows.
\end{proof}

\begin{lemma} \label{jordan}  Let $x \in \gl_n=\gl(V)$ with $\car k \ne 2$. Assume that $x$ has  $r$ Jordan blocks of odd
size.  Let $s$ be the number of Jordan blocks of $x$ on $\Sym^2V$ and $a$ the number of Jordan
blocks on $\wedge^2V$.   Then $s-a=r$.  
 \end{lemma} 
 
 \begin{proof}
 Write $V=V_1 \oplus \cdots \oplus V_m$ where $x$ on $V_i$ is a single Jordan block.   
 Then (as an $x$-module),  $\Sym^2V = 
\left( \oplus_{i < j} V_i \otimes V_j \right) \oplus \left (\oplus_i \Sym^2V_i \right)$ 
  and $\wedge^2 V = \left( \oplus_{i < j} V_i \otimes V_j \right) \oplus \left( \oplus_i \wedge^2V_i\right)$. 
Thus, the difference in the number of Jordan blocks on $\Sym^2 V$ and $\wedge^2 V$ is just
the sum of the differences on $\Sym^2V_i$ and $\wedge^2V_i$ and the result follows
by the previous lemma. 
 \end{proof}

Put $\lambda$ for the highest weight of the natural module of $\so_n$, i.e., $\lambda = \omega_{\lfloor n/2 \rfloor}$ as in Table \ref{luebeck.table}.
We can now show  that $\so_n$ acts generically freely on $L(2\lambda)$ in characteristic not $2$ 
by proving that our standard inequality \eqref{ineq.mother} holds.   (See \cite[Example 10.7]{GG:simple} or \cite[\S4.1]{GurLawther} for another proof that the generic
stabilizer is an elementary abelian $2$-group as a group scheme.)   If $\car k$ does not divide $n$, then $W$ is a summand of the natural representation $V$ with a trivial $1$-dimensional
complement.  If $\car k$ divides $n$, then $\Sym^2V$ is a uniserial module
with trivial head and socle and $W$ the unique nontrivial composition factor.

\begin{lemma}  \label{SO.S2} 
Let $\g=\so_n =\so(V)$ with $n \ge 5$ and
$\car k \ne 2$. 
Set $W=L(2\lambda)$.  If $x \in \g$ is nonzero nilpotent or noncentral semisimple, 
then $\dim x^G + \dim W^x < \dim W$.
\end{lemma}

\begin{proof}  

 If $x$ is semisimple,  by considering weights on $V$, $\Sym^2V$ and $\wedge^2V$,
we see that $\dim (\Sym^2V)^x - \dim (\wedge^2V)^x = \dim V^x $, using that $\car k \ne 2$.   Since
$\dim x^G + \dim (\wedge^2V)^x = \dim G$, we see that 
\[
\dim x^G + \dim (\Sym^2V)^x = \dim G + \dim V^x = \dim \Sym^2V - (\dim V - \dim V^x),
\]
which is at most $\dim \Sym^2V - 2$, because the fixed space of $x$ has codimension at least 2.
Since $L(2\lambda)$ is a summand of $\Sym^2V$ as an $x$-module and $x$
is trivial on a complement, the result follows. (Note that if $\car k$ divides $n$, then $L(2\lambda)$ is not a summand of
$\Sym^2 V$ for $G$.)   

If $x$ is nilpotent,  we argue similarly using the previous lemma.    Note that, by Lemma \ref{independent},
$\dim x^G + \dim (\wedge^2V)^x = \dim G$.   
Thus by  Lemma \ref{jordan}, 
\[
\dim x^G + \dim (\Sym^2V)^x =   \dim G  +  r
= \dim \Sym^2(V) -(n-r)  \le \dim \Sym^2(V) - 2.
\]

Assume $\car k$ divides $n$, for otherwise the result follows.
Note that
$\dim W^x \le \dim (\Sym^2V)^x$ and the result follows unless $r=n-2$  and 
$\dim W^x  = \dim (\Sym^2V)^x$.     The first condition implies that $x$ has 
one nontrivial Jordan block  which must be of size $3$.     In this case,  a trivial calculation gives
$\dim W^x = \dim (\Sym^2V)^x - 2$ and the result follows.
\end{proof} 

%%%%%%%%%%%%%%%%%%%%%%%%%%%%%%%%%%%%%%%%%%%%%%%%%%%%%%%%%%%%%%%%
\section{Example: Vinberg representations} \label{vinberg}

Let $G$ be an algebraic group over a field $k$ and suppose $\theta \in \Aut(G)(k)$ has finite order $m$ not divisible by $\car k$.  Choosing a primitive $m$-th root of unity $\zeta \in \kx$ gives a $\Z/m$-grading $\g = \oplus_{i \in \Z/m} \g_i$ where $\g_i = \{ x \in \g \mid \theta(x) = \zeta^i x \}$.  The sub-scheme $G_0$ of fixed points is smooth, see, for example, \cite[Exercise 2.4.10]{Conrad:red}.  In this section we will assume furthermore that $G$ is semisimple simply connected, in which case $G_0$ is connected reductive \cite[A.8.12]{CGP2} and can be described explicitly using the recipe in \cite[\S8]{St:end}.   Representations $(G_0, \g_1)$ arising in this way are sometimes called Vinberg representations or $\theta$-groups.

\begin{lem} \label{regular.lem}
Let $T$ be a maximal torus in a simple algebraic group $G$ over a field $k$.  Then,  (1) $G = \Sp_{2n}$ for some $n \ge 1$ and $\car k = 2$ or (2) for a generic $t \in \Lie(T)$, the transporter $\{ x \in \Lie(G) \mid [x,t] \in \Lie(T) \}$ equals $\Lie(T)$.
\end{lem}

\begin{proof}
Write $x$ as a sum of an element $x_0 \in \Lie(T)$ and a generator $x_\alpha$ in the root subalgebra for each root $\alpha$.  Choose $t \in \Lie(T)$ generic and suppose $[x, t] \in \Lie(T)$, i.e., $\mathrm{d}\alpha(t) x_\alpha = [x_\alpha, t] = 0$ for all $\alpha$.  If (1) fails, then an exercise with roots as in \cite[Lemma 2.13]{ChaputRomagny} shows that $\mathrm{d}\alpha(t) \ne 0$ for every root $\alpha$, whence the claim.
\end{proof}

\begin{eg}[$m = 2$] \label{vinberg.2}
Suppose $\theta \in \Aut(G)(k)$ has order 2 and acts on a maximal torus $T$ via $\theta(t) = t^{-1}$ for $t \in T$, so $\Lie(T)$ is contained in $\g_1$.  As $\car k \ne 2$, the centralizer in $\Lie(G)$ of a generic element in
$\Lie(T)$ is just $\Lie(T)$
which misses $\g_0$, whence $\g_0$ acts virtually freely on $\g_1$.  More precisely, as a group scheme, the stabilizer in $G_0$ of a generic element of $\Lie(T)$ is the 2-torsion subgroup of $T$.
In this way, if we pick a subgroup $H$ of $G_0$, we conclude that $\lsub$ acts generically freely on $\g_1$.  We now consider examples where this applies; in each case a generic element of $\g_1$ is a regular semisimple element of $\g$, see \cite[\S7]{RLYG} or \cite[\S4.1]{GurLawther}.

(1): Take $G$ to have type $E_6$ and $\theta$ to be an outer automorphism so that $G_0$ is the adjoint group $\PSp_8$ of type $C_4$, compare, for example, \cite[\S5]{GPT}. In that case, $\g_0 = \sp_8$ and $\g_1$ is the Weyl module with highest weight 1000 (the ``spin'' representation).  If $\car k \ne 3$ (and $\ne 2$), then the representation $\g_1$ is irreducible of dimension 42.  

If $\car k = 3$, $\g_1$ has head the irreducible representation of dimension 41 and radical $k = \z(\g)$.  Let $v$ be a regular semisimple element of $\Lie(T) \subset \g_1$.  The stabilizer in $\g_0 = \sp_8$ of the image of $v$ in $\g_1/k$ transports $v$ into $\z(\g)$, and therefore belongs to $\Lie(T) \cap \g_0 = 0$ by Lemma \ref{regular.lem}.  In particular, $\sp_8$ acts generically freely on the irreducible representation $\g_1/k$.

(2): Take $G$ to be $E_8$ and $\theta$ to be such that $G_0$ has type $D_8$.  In this case, $G_0$ is a half-spin group $\Spin_{16}/\mu_2$ and $\g_1$ is the 128-dimensional half-spin representation.  We conclude that $\g_0$ acts generically freely when $\car k \ne 2$.  (Regardless of $\car k$, the generic stabilizer in $G_0$ is $(\Z/2)^4 \times (\mu_2)^4$ as a group scheme, see \cite[Th.~1.2]{GG:spin}.)

(3): Take $G$ to be $E_7$ and $\theta$ to be such that $G_0 = \SL_8/\mu_4$.  In this case, $\g_1$ is the representation $\wedge^4 k^8$, which is generically free for $\car k \ne 2$.  (We provide a stronger result in Prop.~\ref{A.wedge}\eqref{A.wedge.good}.)

(4): Take $G$ to be $\SL_n$ with $\theta(g) = g^{-\top}$, so $G_0 = \SO_n$ and $\g_1$ is the Weyl module with head $L(2\lambda)$ as in Lemma \ref{SO.S2}.
\end{eg}

The representation $\wedge^3 k^9$ of $G_0 = \SL_9/\mu_3$ arises also in this way when $G = E_8$ and $m = 3$, see \cite{Vinberg:tri} for a detailed analysis of the orbits in the case $\car k = 0$.  A generic element of $\g_1$ is regular semisimple as an element of $\g$ as in the references in Example \ref{vinberg.2} (\cite{GurLawther} produces an explicit regular nilpotent element), and we find that $\sl_9$ acts generically freely on $\wedge^3 k^9$. We will provide a stronger result below in Prop.~\ref{A.wedge}\eqref{A.wedge.good}.

\begin{rmk} \label{vinberg.equi}
The setup above can be generalized to accommodate the case where $\car k$ divides $m$.  Instead of an element $\theta \in \Aut(G)(k)$, one picks a homomorphism of group schemes $\mu_m \to \Aut(G)$ defined over $k$.  Again one obtains a $\Z/m$-grading on $\g$ and an action of $\mu_m$ on $G$ such that $G_0$ is smooth.  Some statements about the representation $\g_1$ of $G_0$ from \cite{Vinberg:Weyl} or \cite{Levy:Vinberg} do not hold in this generality.  For example, the representations from Example \ref{vinberg.2}(2) and (3) with $\car k = 2$ and the representation $\wedge^3 k^9$ of $\SL_9/\mu_3$ with $\car k = 3$, are not virtually free for $\g_0$.  This can be seen by computationally verifying that Lemma \ref{nvfree} applies; in each of these three cases the stabilizer of a generic vector is a toral subalgebra whose dimension we list in Table \ref{irred.nvfree}.  Alternatively, for $x \in \g_1$, $x^{[p]}$ is in $\z_{\g_0}(x)$, so finding any $x$ with $x^{[p]}$ not in the kernel of the representation (as is done in 
\cite[Prop.~4.8.2, 4.9.2]{Auld}) suffices to show that the representation is not virtually free.

For the spin representation of $\Sp_8$, 2 is a special prime so is treated in \cite{GG:special}.
\end{rmk}

%%%%%%%%%%%%%%%%%%%%%%%%%%%%%%%%%%%%%%%%%%%%%%%%%%%%%%%%%%%%%%%%
\section{Example: 3rd and 4th exterior powers} \label{wedge.sec}

We now consider the representation $\wedge^e (k^n)$ of $\SL_n$ and its analogues for $\SO_n$ and $\Sp_n$.  Whether or not such representations are virtually free has previously been considered in \cite{Auld} and elsewhere.  We will check here the stronger condition of whether or not inequality \eqref{ineq.mother} holds for $x \in \sl_n$.

\begin{prop} \label{A.wedge}
For the representation $V := \wedge^e(k^n)$ of $\SL_n$ and noncentral $x \in \sl_n$ with $x^{[p]} \in \{ 0, x \}$, we have:
\begin{enumerate}
\item \label{A.wedge.good} If (a) $e = 3$ and $n \ge 10$; (b) $e = 3$, $n = 9$, and $\car k \ne 2, 3$; (c) $e = 4$ and $n \ge 9$; or (d) $e = 4$, $n = 8$, and $\car k \ne 2$, then $\dim x^{\SL_n} + \dim V^x < \dim V$ and $\Lie(\SL_n/\mu_{\gcd(e,n)})$ acts generically freely on $V$.
\item \label{A.wedge.bad} If (a) $e = 3$, $n = 9$, and $\car k = 2, 3$ or (b) $e = 4$, $n = 8$, and $\car k = 2$, then $\dim x^{\SL_n} + \dim V^x \le \dim V$.
\end{enumerate}
\end{prop}

\begin{proof}
Suppose $x^{[p]} = 0$.  The case where $x^{[p]} = x$ follows from it by \cite[Lemma 4.2]{GG:large}.

Put $n_0 = 16$ if $e = 3$ and $n_0 = 10$ if $e = 4$.  If $n > n_0$, then $\dim V = \binom{n}{e} > 2.25n^2 \ge b(\SL_n)$, and \eqref{ineq.mother} holds by \cite{GG:large}.

So suppose $n \le n_0$.  We calculate $\dim x^{\SL_n}$, which does not depend on $\car k$, using the well-known formulas from, for example, \cite[p.~39]{LiebeckSeitz}.  For the other term in \eqref{ineq.mother}, $\dim V^x$, we view $V$ as a representation of $\SL_2$ where a nilpotent element acts as $x$ on $V$.  Arguing as in \cite[\S3.4]{McN:exterior}, we find that if $\car k > en$, then the Jordan form of $x$ acting on $V$ is the same as in characteristic zero.  Therefore, it suffices to check the inequality over $\F_p$ for $2 \le p < en$ and for some $p$ larger than $en$.  This is quickly done via computer.  For the convenience of the reader, Table \ref{A.wedge.table} lists the partitions corresponding to nilpotent $x$ for which we have equality in \eqref{A.wedge.bad}.  In case $\gcd(e,n) = 1$, this shows that $\sl_n$ acts generically freely on $\wedge^e k^n$.  For each $n \le n_0$ with $\gcd(e,n) > 1$, we verify that $\Lie(\SL_n / \mu_{\gcd(e,n)})$ acts generically freely using Magma.
\end{proof}

\begin{table}[hbt]
\begin{tabular}{cc|lrr}
representation&$\car k$&partition of $x$&$\dim x^G$&$\dim V^x$ \\ \hline\hline
$\wedge^3 \sl_9$&2&$(2^4, 1)$&40&44 \\ \hline
$\wedge^3 \sl_9$&3&$(9)$&72&12 \\
&&$(3^3)$&54&30\\ \hline
$\wedge^4\sl_8$&2&$(8)$&56&14 \\
&&$(4^2)$&48&22 \\
&&$(2^4)$&32&38 
\end{tabular}
\caption{Complete list of nilpotent elements $x$ from Proposition \ref{A.wedge}\eqref{A.wedge.bad} where equality holds.} \label{A.wedge.table}
\end{table}

% \begin{itemize}
% \item For $\sl_9$ on $\wedge^3(k^9)$ when $\car k = 2$: $(2^4, 1)$. 
% \item For $\sl_9$ on $\wedge^3(k^9)$ when $\car k = 3$: $(9)$ and $(3^3)$.
% \item For $\sl_8$ on $\wedge^4(k^8)$ when $\car k = 2$: $(8)$, $(4^2)$, and $(2^4)$. $\hfill\qedhere$
% \end{itemize}

\subsection*{Trivectors and $\SO_n$} 
Consider now $\SO_n$ with $n \ge 9$.  The representation $\wedge^3(k^n)$ is a fundamental Weyl module and is irreducible  if $\car k \ne 2$, see for example \cite[II.8.21]{Jantzen} and \cite[Remark 3.4]{McN:exterior}.  

\begin{prop} \label{BD.wedge3}
For $\SO_n$ with $n \ge 9$ (over any field $k$) and $V := \wedge^3(k^n)$, the inequality \eqref{ineq.mother} holds for all nonzero $x \in \so_n$ with $x^{[p]} \in \{ 0, x\}$, and $\g$ acts generically freely on $V$.
\end{prop}

\begin{proof}
Under the tautological inclusion $\SO_n \hookrightarrow \SL_n$, suppose the inequality holds for $x$ viewed as an element of $\sl_n$.  Then as $\dim x^{\SO_n} \le \dim x^{\SL_n}$, the inequality holds also for $x$ as an element of $\so_n$, completing the proof in case $n \ge 10$, or $n = 9$ and $\car k \ne 2, 3$ (Prop.~\ref{A.wedge}).

So suppose $n = 9$ and $\car k = 2$ or $3$.
Write $y$ for the image of $x$ in $\sl_n$ if $x$ is nilpotent, and for the image of the nilpotent specialization of $x$ as in the proof of Prop.~\ref{A.wedge} if $x$ is toral.  As in the previous paragraph, we are done if the inequality holds for $y$, and therefore we may assume that $y$ has partition $(2^4, 1)$ or $(3^3)$ as in the Table \ref{A.wedge.table}.  In either of these cases, we have $\dim x^{\SO_9} + \dim V^x \le 32 + \dim V^y \le 76 < \dim V$, completing the proof. 
\end{proof}

\subsection*{Trivectors and $\Sp_{2\ell}$}
The natural representation of $\Sp_{2\ell}$ has an invariant alternating bilinear form $b$.  The subspace $V(\omega_{\ell-2})$ of $\wedge^3 k^{2\ell}$ spanned by those $v_1 \wedge v_2 \wedge v_3$ with $b(v_i, v_j) = 0$ for all $i, j$ is a submodule of dimension $\binom{2\ell}{3} - 2\ell$; it is the Weyl module with highest weight $\omega_{\ell-2}$, see \cite[\S1]{GowKleshchev}.  It is irreducible, i.e., $V(\omega_{\ell-2}) = L(\omega_{\ell-2})$, if and only if $\ell - 1$ is nonzero in $k$; otherwise $V(\omega_{\ell-2})$ has socle the natural module $k^{2\ell}$ and head $L(\omega_{\ell-2})$ \cite[Th.~2(i)]{PremetSup}.

\begin{prop} \label{C.wedge3}
Continue the 
notation of the preceding paragraph and suppose that $p := \car k > 2$.  If
\begin{enumerate}
\item $\ell \ge 7$ or $\ell = 5$; or 
\item $\ell = 6$ and $p \ne 5$; or
\item \label{C.sp8} $\ell = 4$ and $p \ne 3$, 
\end{enumerate}
then for $V := V(\omega_{\ell-2})$ or $L(\omega_{\ell-2})$, inequality \eqref{ineq.mother} holds for all nonzero $x \in \sp_{2\ell}$ with $x^{[p]} \in \{ 0, x \}$.  In these cases, and also when $(\ell, p) = (6, 5)$,  $\sp_{2\ell}$ acts generically freely on $V$.
\end{prop}

In the case $\ell = 4$ and $\car k = 3$, $\sp_8$ does not act generically freely on $V$, see Example \ref{C.0100.3}.

\begin{proof}
If $\ell > 6$, then $\dim V(\omega_{\ell-2}) \ge \dim L(\omega_{\ell-2}) \ge \binom{2\ell}3 - 4\ell > b(\Sp_{2\ell})$, and the conclusion holds by \cite[Th.~12.2]{GG:large}.   So suppose $\ell = 4$, 5, or 6.  In particular, $\ell - 1$ is not zero in $k$ and $V = V(\omega_{\ell-2}) = L(\omega_{\ell-2})$.

First suppose that $x \in \sp_{2\ell}$ has $x^{[p]} = 0$. 
If $\ell = 5$ or 6, we have $\dim L(\omega_{\ell-2})^x \le \dim V(\omega_{\ell-2})^x \le \dim (\wedge^3 k^{2\ell})^x$ and one checks that $\dim x^{\Sp_{2\ell}} + \dim (\wedge^3 k^{2\ell})^x < \dim L(\omega_{\ell-2})^x$, which need only be done for small characteristics as in the proof of Proposition \ref{A.wedge} and therefore amounts to a computer calculation.
If $\ell = 4$ (and $\car k > 3$), then $\wedge^3 k^8$ is a direct sum of $V$ and $k^8$, so $\dim V^x = \dim (\wedge^3 k^8)^x - \dim (k^8)^x$ and the same computer calculations verify \eqref{ineq.mother}.

For $x$ toral, we appeal to Lemma \ref{auld.popov}.

Lemma \ref{ineq.mother} gives that $\sp_{2\ell}$ acts generically freely, except in the case $\ell =6$ and $\car k =5$ which we verify using Magma.
\end{proof}

%%%%%%%%%%%%%%%%%%%%%%%%%%%%%%%%%%%%%%%%%%%%%%%%%%%%%%%%%%%%%%%%
\section{Theorem \ref{MT} for restricted highest weights} \label{restricted.sec}

In this section, we will prove the following by deducing it from what has come before.

\begin{prop} \label{MT.restricted}
Theorem \ref{MT} holds if $\car k$ is not special for $G$ and the highest weight of $\rho$ is restricted.
\end{prop}

In addition to the case of type $A_1$, many other cases were handled in \cite{GG:large}.  Corollary B in ibid.\ reduces us to considering the following:
\begin{enumerate}
\item $G$ has type $A_\ell$ for $2 \le \ell \le 15$;
\item $G$ has type $B_\ell$ or $C_\ell$ with $2 \le \ell \le 11$; or
\item $G$ has type $D_\ell$ for $4 \le \ell \le 11$.
\end{enumerate}
Theorem A in ibid.\ allows us to further assume that $\dim V \le b(G)$ for $b(G)$ as in Table \ref{classical.table}.  Therefore, $V$ appears in tables A.6--A.48 in \cite{luebeck}.  (Note that the search space remains infinite: while there are only finitely many possibilities for $G$ and for the highest weight of $V$, we have not exhibited any upper bound on $\car k$.) 

If $\dim V < \dim G - \dim \z(\g)$, then certainly $\g$ cannot act virtually freely on $V$.  If $\dim G \ge \dim V \ge \dim G - \dim \z(\g)$, then examining the tables shows that $V$ is the irreducible representation with highest weight the highest root, which is not virtually free as in \cite[Example 3.4]{GG:large}.  Therefore we assume for the rest of this section that $\dim G < \dim V \le b(G)$.  We check, for each such $V$, that $\g$ acts generically freely or that $(G, V, \car k)$ appears in Table \ref{irred.nvfree}.

\subsection*{Type $A$} For $A_\ell$, we consider $2 \le \ell \le 15$.  

Proposition \ref{A.wedge} treats the representations $\wedge^3 k^{\ell+1}$ and $\wedge^4 k^{\ell+1}$ of $G$ of type $A_\ell$ apart from a few cases.  
For $\wedge^3 k^9$ when $\car k = 2$, $\sl_9$ acts generically freely on $V$ by \cite[Prop.~4.8.3]{Auld}, by reasoning as in \S\ref{vinberg}, or as can be checked in Magma, despite the failure of inequality \eqref{ineq.mother}.   
The representations $\wedge^3 k^9$ of $G = \SL_9/\mu_3$ when $\car k = 3$ or $\wedge^4 k^8$ of $G = \SL_8/\mu_4$ when $\car k = 2$ are not virtually free, see Remark \ref{vinberg.equi}.

We refer to \cite[Th.~4.3.2]{Guerreiro} for the representation of $A_\ell$ $(2 \le \ell \le 9)$ with highest weight $0\cdots03$ (and $\car k > 3$ so it is restricted); of $A_3$ with highest weight 004 and dimension 35; of $A_3$ with highest weight 102 and $\car k \ne 5$; of $A_2$ with highest weight 04 and dimension 15; and of $A_2$ with highest weight 13 or 22 with $\car k = 5$.  

For the representation of $A_\ell$ $(3 \le \ell \le 9)$ with highest weight $0\cdots011$ with $\car k = 3$, we verify using Magma with $G = \SL_{\ell+1}/\mu_{\gcd(\ell+1,3)}$.  (Guerreiro checked that $\SL_{\ell+1}$ acts virtually freely, see Claim 12 on p.~97 of \cite{Guerreiro}.)

We refer to \cite{Auld} to see that the following are virtually free: the representation of $A_\ell$ ($\ell = 3, 4, 5$) and $\car k \ne 3$ with highest weight $0\cdots011$ (\S4.5); 
the representation of $A_2$ with highest weight 12 and dimension 15 when $\car k \ne 2$ (\S4.1);
the representation of $A_3$ with highest weight 102 and dimension 32 when $\car k = 5$ (\S4.2);
the representation of $A_4$ with highest weight 0101 and dimension 40 or 45 (\S4.6);
the representation of $A_4$ with highest weight 0200 and dimension 45 or 50 when $\car k \ne 2$ (\S4.7);
the representation of $A_4$ with highest weight 0110 and dimension 51 when $\car k = 3$ (\S4.4);
and the representation of $A_5$ with highest weight 01001 and dimension 78 when $\car k = 5$ (\S4.6).

The representation of $A_3$ with highest weight 020 and $\car k \ne 2$ is virtually free by Lemma \ref{SO.S2}.

\subsection*{Types $B$ and $D$} For $G$ of type $D_\ell$ with $4 \le \ell \le 11$ or $B_\ell$ with $2 \le \ell \le 11$ and $\car k \ne 2$, the representation with highest weight $0\cdots02$ is handled by Lemma \ref{SO.S2}.  

The (half) spin representations of $B_\ell$ for $7 \le \ell \le 11$ and $D_\ell$ for $\ell = 9, 10, 11$  (for $G = \Spin_n$ for $n = 15$, 17, 18, 19, 21, 22, 23 and $G = \Spin_n/\mu_2$ when $n = 20$) are generically free.  For $G = \Spin_{16}/\mu_2$, the half-spin representation has generic stabilizer $(\Z/2)^4 \times \mu_2^4$ as a group scheme, so it is a generically free representation of $\g$ when $\car k \ne 2$ and is not generically free when $\car k = 2$.  For these results, see \cite{GG:spin}.  

The representation $\wedge^3(k^n)$ of $\SO_n$ with $n = 9, \ldots, 13$ (i.e., $B_4$, $B_5$, $B_6$, $D_5$, $D_6$) with $\car k \ne 2$ is generically free for $\so_n$ by Proposition \ref{BD.wedge3}.  When $\car k = 2$ (and $G$ has type $D$), $\wedge^3(k^n)$ is reducible with irreducible quotient $L(\omega_{\frac{n}2-2})$.  For $\SO_{10}$, we verify with Magma that $L(\omega_3)$ is generically free.  For $n = 12, 14, \ldots$, $\dim L(\omega_{\frac{n}2-2}) > b(\SO_n)$.

The representations of $B_2$ with highest weight 11 and $B_3$ with highest weight 101 are handled in Examples \ref{B.11} and \ref{B.101}.

The representation of $B_3$ ($\so_7$) with highest weight 200 and dimension 35 when $\car k \ne 2$ appears as a summand in $\Sym^2 X$ for $X$ the (8-dimensional) spin representation; we have $\Sym^2 X \cong V \oplus k$.  The action on $V$ factors through the action of $\so_8$ as in Lemma \ref{SO.S2}, whence we have the inequality for $V$.  Similarly the representation of $\so_9$ with highest weight 2000 and dimension 126 is generically free because it factors through the generically free representation of $D_5$ as in Example \ref{D.20000}.

We refer to \cite[Th.~4.3.3]{Guerreiro} for the representations of $B_4$ ($\spin_9$) with highest weight 1001 and dimension 112 or 128; of $B_3$ with highest weight 011 and dimension 63 and $\car k = 3$; of $B_3$ with highest weight 110 and dimension 64 and $\car k = 5$; and of $B_2$ with highest weights 30, 12, 03, or 21.

The representation of $D_4$ with highest weight 1001 has dimension greater than $\dim D_4$ and is generically free as in Example \ref{D.1001}.

For the representation of $D_5$ with highest weight 10001 of dimension 144 with $\car k \ne 2, 5$, \cite[Th.~4.3.5]{Guerreiro} proves it is generically free.  If $\car k = 2$, that representation has dimension $144 > b(G)$ and the inequality holds.  If $\car k = 5$, one checks with Magma that a random vector has zero-dimensional stabilizer.

The representation of $D_5$ with highest weight 20000 of dimension 126 is generically free by Example \ref{D.20000}.

\subsection*{Type $C$} Type $C$ is similar to types $B$ and $D$.  We consider $3 \le \ell \le 11$.  Excluding those $V$ with $\dim V > b(G)$ reduces us further to $3 \le \ell \le 6$. 

The only case for which we refer to \cite{Guerreiro} is type $C_3$ with highest weight $011$ of dimension 50 with $\car k = 3$ (Th.~4.3.4), which can also be checked using Magma.  The representation of $C_5$ with highest weight 10000 is generically free by Example \ref{C.10000}.

We use Magma to verify that a random vector has trivial stabilizer when $\car k = 3$ for $C_5$ with highest weight $01000$ and dimension 121.

The representation $V$ of $C_4$ with highest weight 1000, which was treated in Example \ref{vinberg.2}.  It has $\dim V > \dim C_4$ and $V$ is generically free.

The representation of $\Sp_{2\ell}$  with highest weight $00 \cdots 0100$ with $\ell = 4, 5, 6$ is generically free by Proposition \ref{C.wedge3}, except for $C_4$ in characteristic 3, see Example \ref{C.0100.3}.

This completes the proof of Proposition \ref{MT.restricted}.$\hfill\qed$

%%%%%%%%%%%%%%%%%%%%%%%%%%%%%%%%%%%%%%%%%%%%%%%%%%%%%%%%%%%%%%%%%%%%%%%%%
\section{Theorem \ref{MT} for some tensor decomposable representations}

Next we treat a family of irreducible but tensor indecomposable representations.  Implicitly, we fix a pinning of $G$, which includes a choice of maximal torus $T$; the lattice $T^*$ of characters $T \to \Gm$ is contained in the weight lattice $P$ (and $T^*$ is identified with $P$ when $G$ is simply connected).  In this section we will prove:

\begin{prop} \label{MT.tensor}
Theorem \ref{MT} holds if $\car k$ is not special and the highest weight $\la$ of $\rho$ satisfies $\la = \la_0 + p\la_1$ where $p = \car k \ne 0$, $\la_0$ and $\la_1$ belong to $T^*$, and $\la_0$ is restricted.
\end{prop}

\begin{lem} \label{SL.tensor}
Let $G$ be a semisimple algebraic group.  For every representation $W$ of $G$, $\g$ acts virtually freely on $W \ot W^{[p]^i}$ and $W \ot (W^*)^{[p]^i}$ for all $i > 0$.
\end{lem}

\begin{proof}
Put $V := W \ot W^{[p]^i}$ or $W \ot (W^*)^{[p]^i}$.

Suppose first that $G = \SL_n$ and $W$ is the natural module.  The representation of $\sl_n$ on $V$ is equivalent to a direct sum of $\dim W$ copies of the natural module, i.e., is equivalent to $\sl_n$ acting on $n$-by-$n$ matrices by left multiplication.  A generic matrix $v$ is invertible, so the generic stabilizer $(\sl_n)_v$ is zero.  (We remark that the group $\SL_n$ has finitely many orbits on $\P(V)$ \cite[Lemma 2.6]{GLMS}.)

Otherwise, the representation $G \to \GL(V)$ factors through $\SL(W) \to \GL(V)$, because $G$ is semisimple, and the previous paragraph shows that $\sl(V)$ acts virtually freely.
\end{proof}

Note that, in the lemma, the inequality \eqref{ineq.mother} need not hold.  Specifically, a root element $x \in \sl_n$ has $\dim x^{\SL_n} = 2(n-1)$ and kernel of dimension $n-1$ on the natural module, so we find $\dim x^{\SL_n} + \dim V^x = \dim V + n - 2$ for $V$ a sum of $n$ copies of the natural module. 

\begin{eg} \label{SO.tensor} 
Consider now $\SO_n$ for $n \ge 3$ and suppose that $\car k \ne 2$ or $n$ is even.  Take $V_c$ to be a direct sum of $c$ copies of the natural module $V_1$.  Let $v \in V_c$ be a vector each of whose $c$ components is a generic element of $V_1$; in particular, the $\SO_n$-invariant quadratic form $q$ is nonzero on each component of $v$.  The $c$ components generate a $c$-dimensional subspace $U$ of $V_1$ on which the bilinearization of $q$ is nondegenerate if $\car k \ne 2$ or $c$ is even, or has a 1-dimensional radical on which $q$ does not vanish if $\car k = 2$ and $c = n -1$.  

Therefore, if $c = n -1$, an element of $\so_n$ that annihilates $U$ is zero on $V_1$, hence $\so_n$ acts generically freely on $V_{n-1}$.  If $c = n-2$, then an element of $\so_n$ that annihilates $U$ belongs to $\so(U^\perp)$ for $U^\perp$ the 2-dimensional subspace of $V_1$ orthogonal to $U$ with respect to the bilinear form, i.e., the stabilizer of a generic $v \in V_{n-2}$ is a rank 1 toral subalgebra of $\so_n$.  
(In case $\car k = 0$, Table 2 of \cite{Elashvili:can} summarizes this and many similar examples.  See also \cite{BGS} for more general arguments in a similar vein.)

Finer results can be proved.  For example, suppose $\car k = 2$ and $n \ge 8$ is even.  We have already observed that $\so_n$ acts generically freely on $V_{n-1}$, but more is true: \emph{the inequality \eqref{ineq.mother} holds for noncentral $x \in \go_n$ such that $x^{[2]} \in \{ 0, x \}$.}   If $x^{[2]} = 0$, then,  as a linear transformation on $V_1$, $x$ has even rank $r \le n$ and $\dim x^{\SO_n} \le r(n-r)$, so 
\[
\dim V^x_{n-1} + \dim x^{\SO_n} = (n-r)(n-1+r).
\]
This is less than $\dim V_{n-1}$ since $r \ge 2$.
In case $x^{[2]} = x$, the inequality is verified by arguing as in the proof of \cite[Cor.~10.6]{GG:large}.
\end{eg}

\begin{eg}[$A_3$] \label{A3.eg}
Suppose $G$ has type $A_3$, $p := \car k \ne 0$, and consider the irreducible representation $V$ with highest weight $\la = \omega_2 + p^e \omega_1$ for $e \ge 1$; we take $G$ to be the member of the isogeny class that acts faithfully on $V$.  The composition $\SL_4 \to G \to \GL(V)$ is, as a representation of $\SL_4$,  $L(\omega_1)^{[p]^e} \ot L(\omega_2) = (k^4)^{[p]^e} \otimes \wedge^2 (k^4)$.  As a representation of the Lie algebra $\sl_4$, $V$ is a sum of 4 copies of $\wedge^2 (k^4)$.  

If $p$ is odd, then $G$ is $\SL_4$.    The differential of the isogeny $\SL_4 \to \SL_4/\mu_2 = \SO_6$ identifies $\sl_4$ with $\so_6$, and the action of $\so_6$ on a sum of 4 copies of its natural represention is not generically free by Example \ref{SO.tensor}, so $V$ is not generically free for $\sl_4$.

If $p = 2$ and $e \ge 2$, then $V$ is faithful as a representation of $\SO_6$, and the same argument shows that $V$ is not generically free for $\so_6$.  Nonetheless, $\sl_4$ does act virtually freely.  This can be seen by noting that semisimple elements of $\so_6$ have eigenvalues that come in pairs (say, $\la_1, \la_2, \la_3$ each occurring twice), the image of $\sl_4$ in $\sl_6$ only contains those with $\la_1 + \la_2 + \la_3 = 0$ (because the sum on the left side is $\SL_4$ invariant \cite[Example 8.5]{GG:simple} and the image of $\sl_4$ is the unique codimension-1 $\SL_4$-submodule of $\so_6$), and the elements of the generic stabilizer $\so_2$ in $\so_6$ have two of $\la_1, \la_2, \la_3$ equal to zero.  Alternatively, Magma verifies that $\sl_4$ acts virtually freely.

Finally, if $p = 2$ and $e = 1$, then $G = \PGL_4$ and we claim that $\pgl_4$ acts generically freely on $V$.
Write down the map $\pgl_4 \to \gl(V)$ explicitly as follows.   Fixing a pinning for $\SL_4$ and bases for $k^4$ and $\wedge^2 k^4$ consisting of weight vectors, we can write down the image in $\gl(V)$ of the generator of each of the root subalgebras of $\sl_4$.  Now, the image of $\sl_4$ in $\pgl_4$ has codimension 1, corresponding to the statement that the weight lattice for $A_2$ is generated by the root lattice and the fundamental weight $\omega_1$, so $\pgl_4$ is generated by the image of $\sl_4$ and a semisimple element $h$ corresponding to $\omega_1$ in the sense that $hv_\delta = \qform{\delta, \omega_1} v_\delta$ for every weight vector $v_\delta$ of weight $\delta$ in every representation of $\PGL_4$; this describes the image of $h$ in $\gl(V)$.  From this, Magma verifies that $\pgl_4$ acts generically freely on $V$.
\end{eg}

\begin{proof}[Proof of Proposition \ref{MT.tensor}]
By hypothesis, $V \cong L(\la_0) \otimes L(\la_1)^{[p]}$ \cite[I.3.16]{Jantzen} and $\la_0 \ne 0$ (Lemma \ref{faith}\eqref{faith.0}).
If $\la_1 = 0$ then $\la$ is restricted and we are done by Proposition \ref{MT.restricted}, so assume that $\la_1 \ne 0$.
By Lemma \ref{faith}\eqref{faith.1}, $\dim V > \dim G$; our task is to show that $V$ is generically free if and only if $(G, \car k, V)$ does not appear in Table \ref{irred.nvfree}.

As $\g$ acts trivially on $L(\la_1)^{[p]}$, the representation $V$ of $\lie$ is the same as a sum of $\dim L(\la_1)$ copies of $L(\la_0)$.
Let $m(G)$ be the dimension of the smallest nonzero irreducible representation of $G$ with restricted highest weight as in Table \ref{mG.table}.  If $\dim L(\la_0) > b(G)/m(G)$ for $b(G)$ as in Table \ref{classical.table}, then $\dim V > b(G)$ and $\g$ acts virtually freely on $V$ by part I.  In particular, if $m(G)^2 > b(G)$ --- as is true for $G$ exceptional --- we are done.

If $\dim L(\la_0) = m(G)$, then $V$ (considered as a $\g$-module) contains $L(\la_0) \ot L(\la_0)^{[p]}$ as a summand, and we are done by Lemma \ref{SL.tensor}.  Therefore, it remains to inspect $\oplus^{m(G)} L(\la_0)$ for those nonzero restricted dominant weights $\la_0$ with 
\begin{equation} \label{dim.bd}
m(G) < \dim L(\la_0) \le b(G)/m(G).  
\end{equation}
We proceed case by case, where the possibilities for $\la_0$ are enumerated in \cite{luebeck}.  We find very few possibilities, reflecting the fact that the bounds in \eqref{dim.bd} both grow linearly in the rank of $G$.

\subsubsection*{Type $B$}
For $G$ of type $B_\ell$ with $\ell \ge 3$, the constraint \eqref{dim.bd} reduces us to consider the case where $G$ has type $B_3$ and $L(\la_0)$ is the 8-dimensional spin representation.  Then $L(\la_0)$ factors through $\Spin_8$ as a vector representation, and we apply Example \ref{SO.tensor} to see that the generic stabilizer in $\so_8$ is trivial and therefore the same is true for $\spin_7$.

\subsubsection*{Type $C$}
For $G$ of type $C_\ell$ with $\ell \ge 2$, the dimension bounds reduce us to considering
$C_2$ where $L(\la_0)$ is the 5-dimensional fundamental irreducible representation, i.e., $\g = \sp_4 = \so_5$ acting on a sum of four copies of
the 5-dimensional module.  This action is generically free by Example \ref{SO.tensor}.

\subsubsection*{Type $D$}
For $G$ of type $D_\ell$ with $\ell \ge 4$, the unique dominant weight $\la_0$ that must be considered is for type $D_5$ with 
$\car k \ne 2$ and $L(\la_0)$ a half-spin representation, so $G = \Spin_{10}$ and we may take $V = \oplus^{10} L(\la_0)$.
Now $\dim L(\la_0)^x \le 12$ by \cite[Prop.~2.1(i)]{GG:spin} and   \eqref{ineq.mother} holds  unless
$\dim x^G =40$ and $x$ is regular.  Thus, it suffices to show that $\dim L(\la_0)^x < 12$ for $x$ regular.
By passing to closures it suffices to take $x$ regular nilpotent.  Let $y$ be a nilpotent element in the Levi
of type $\gl_5$ with two Jordan blocks of size $5$ (in the natural representation of $\so_{10}$).   As we have seen,
two conjugates of $y$ generate an $\sl_5$.  On the half-spin module, $\sl_5$ has one trivial submodule
and two nontrivial irreducible composition factors.  Thus, $\dim L(\la_0)^y \le 8$ and since $y$ is
in the closure of $x^G$, the same is true of $x$ and $V$ is generically free for $\g$.

\subsubsection*{Type $A$}
For type $A_\ell$ with $\ell \ge 2$, the dimension bounds \eqref{dim.bd} reduce us to the following cases: 
\begin{enumerate}
	\renewcommand{\theenumi}{\roman{enumi}}
	\item \label{tensor.2} $A_2$, where $\car k \ne 2$ and $L(\la_0) = \Sym^2(k^3)$;
	\item \label{tensor.3} $A_3$, where $L(\la_0) = \wedge^2(k^4)$; and
	\item \label{tensor.4} $A_4$, where $\car k \ne 2$ and $L(\la_0) = \wedge^2(k^5)$.
\end{enumerate}
Case \eqref{tensor.3} is handled in Example \ref{A3.eg}.

For cases \eqref{tensor.2} and \eqref{tensor.4}, $G = \SL_{\ell+1}$.  We verify inequality \eqref{ineq.mother} for nonzero $x \in \g$ with $x^{[p]} = 0$.  This will verify it also for noncentral toral $x \in \g$ \cite[Lemma 4.2]{GG:large}, whence $\g$ acts generically freely on $V$.

Consider the most complicated case, \eqref{tensor.4}.  For $x \in \sl_5$ a root element, i.e., nilpotent with partition $(2,1^3)$, we have $\dim x^G = 8$ and $\dim L(\la_0)^x = 7$, and $8+5\cdot 7 = 43 < 50$.  For $x$ nilpotent with partition $(2,2,1)$, we have $\dim x^G = 12$ and $\dim L(\la_0)^x = 6$, and $12 + 5 \cdot 6 = 42 < 50$.  

For $x$ nilpotent with partition $(3,1^2)$, we have $\dim L(\la_0)^x = 4$ and $\dim x^G \le \dim G - \rank G = 20$.  Consequently, for every nilpotent $y \in \sl_5$ such that $x \in \overline{y^G}$, we have:
\[
\dim y^G + \dim V^y \le (\dim G - \rank G) + m(G) \cdot \dim L(\la_0)^x < \dim V.
\]
Thus we have verified the inequality \eqref{ineq.mother} for every nonzero nilpotent in $\sl_5$.

Finally consider case \eqref{tensor.2}.   There are two classes of nilpotent elements.  If 
$x$ is a root element, then $\dim L(\la_0)^x =3$ and $\dim x^G=4$.  If $x$ is regular, then
$\dim L(\la_0)^x =2$ and $\dim x^G=6$.  In both cases, the inequality \eqref{ineq.mother}
holds. 
\end{proof}

%%%%%%%%%%%%%%%%%%%%%%%%%%%%%%%%%%%%
\section{Conclusion of proof of Theorem \ref{MT}} \label{final.sec}

We now complete the proof of Theorem \ref{MT}, assuming $\car k$ is not special.  Write the highest weight $\la$ of $V$ as $\la = \la_0 + p \la_1$ for $\la_0, \la_1$ dominant weights (not necessarily in $T^*$) and $\la_0$ restricted.

Put $\pi \!: \Gt \to G$ for the simply connected cover.  If $G$ is itself simply connected, then we are done by Proposition \ref{MT.tensor}.  Thus we are also done if $\dpi$ is surjective (i.e., if $\ker \dpi = 0$), and we may assume that the finite group scheme $\ker \pi$ is not smooth and has exponent divisible by $p$, reducing us to the following cases: $G = \SL_n / \mu_m$ where $p \mid m$ (and $n \ge 3$), $G$ has type $D_\ell$ and $p = 2$, $G$ is adjoint of type $E_6$ and $p = 3$, or $G$ is adjoint of type $E_7$ and $p = 2$.

Suppose that $\la_0 = 0$.  The composition $\drho \, \dpi$ is the representation $L(\la_1)^{[p]}$ of $\Gt$, whence $\gt$ acts trivially on $V$, so $V$ is not faithful.
% and $\g$ acts virtually freely by \cite[Lemma 10.1]{GG:large}, so we are done in that case.

On the other hand, the case where $\la_1 = 0$ is done by Prop.~\ref{MT.restricted}, so we may assume that $\la_0$ and $\la_1$ are both nonzero.

Now $\la$ vanishes on $\ker \pi$ (because $\la \in T^*$) and $p \la_1$ vanishes on the $p$-torsion in $\ker \pi$, so it follows that $\la_0$ vanishes on the $p$-torsion in $\ker \pi$.  Put $m_p(G)$ for the minimum of $\dim L(\mu)$ as $\mu$ ranges over nonzero restricted dominant weights such that $\ker \mu$ has exponent divisible by $p$; the value of $m_p(G)$ is listed in Table \ref{mpG.table}.  The pullback $\rho\,\pi$ of $\rho$ is the representation $L(\la_0) \ot L(\la_1)^{[p]}$ of $\Gt$, so $\dim V \ge m_p(G) \, m(G)$.

\begin{table}[hbt]
\begin{tabular}{ccrr}
type $G$&$p$&$m_p(G)$&$m(G)$ \\ \hline
$A_\ell$ (odd $\ell \ge 3$)&2&$\binom{\ell+1}{2}$&$\ell+1$ \\
$A_\ell$ ($\ell \ge 2$)&odd $p \mid \ell+1$&$(\ell+1)^2-2$&$\ell+1$ \\
$D_\ell$ ($\ell \ge 4$) &2&$2\ell$&$2\ell$ \\
$E_6$&3&77&27 \\
$E_7$&2&132&56
\end{tabular}
\caption{Value of $m_p(G)$ for various $p$ and $G$.} \label{mpG.table}
\end{table}

In particular, if $m_p(G) \, m(G)$ is greater than $b(G)$, we are done by the main result of part I.  This handles the cases where $G$ has type $E_6$ or $E_7$, or type $A_\ell$ when $p$ is odd.

\begin{lem} \label{tensor}
Consider representations $V$ and $W$ of a Lie algebra $L$.  For nilpotent $x \in L$, $\dim (V \ot W)^x \le (\dim V^x) (\dim W)$.
\end{lem}

\begin{proof}
Put $\psi \!: L \to \gl(V)$ and $\zeta \!: L \to \gl(W)$ for the two actions.
For each $t \in k$, $t \zeta$ is a representation of the Lie algebra $kx$; since $x$ is nilpotent the ones with $t \ne 0$ are all equivalent.    Therefore, writing  $U_t$ for the representation $\psi \ot (t \zeta)$ --- so $U_1 = V \ot W$ --- the dimension of $(U_t)^x$ is constant for $t \ne 0$.  Now $U_0$ is a direct sum of $\dim W$ copies of $(V, \psi)$, so $\dim (U_0)^x = (\dim V^x)(\dim W)$.  On the other hand, by upper semicontinuity of dimension, $\dim (U_0)^x \ge \dim (U_t)^x$ for $t \ne 0$.
\end{proof}

\subsubsection*{Type $D_\ell$} Suppose now that $G$ has type $D_\ell$ and $\car k = 2$, in which case $m_p(G) m(G) = 4\ell^2 = b(G)$ and we are done unless $\dim L(\la_0) = \dim L(\la_1) = 2\ell$ as representations of $\Gt$.  

If $\ell > 4$, the only restricted irreducible representation of $\Gt$ with restricted highest weight and of dimension $2\ell$ is the vector representation $\Spin_{2\ell} \to \SO_{2\ell}$ with highest weight $\omega_\ell$ or a Frobenius twist of it, hence $\la$ is an odd multiple of $\omega_\ell$.  In particular, $\omega_\ell \in T^*$, so $\la_0, \la_1 \in T^*$ and we are done by Proposition \ref{MT.tensor}.

For $\ell = 4$, the representations $L(\omega_i)$ with $i = 1, 3, 4$ of $\Spin_8$ all have dimension 8, and up to graph automorphism we are left with considering the case $\la = 2^e \omega_1 + \omega_4$ for some $e \ge 1$.  Thus we may view $G$ as $\SO_8$, and the pullback to $\Spin_8$ of $V$ is the natural representation $k^8$ of $\SO_8$ (with highest weight $\omega_4$) tensored with a Frobenius twist of a half-spin representation; as a representation of $\SO_8$ we find $k^8 \ot L(2^e \omega_1)$.

Arguing as in Example \ref{SO.tensor}, a square-zero $x \in \so_8$ has even rank $r \le 4$, $\dim x^{\SO_n} \le r(8-r)$ and $\dim (k^8)^x = 8 - r$, so $\dim (k^8 \ot L(2^e \omega_1))^x \le 8(8-r)$ (Lemma \ref{tensor}) and $\dim x^{\SO_n} + \dim V^x \le 64-r^2 < \dim V$, verifying \eqref{ineq.mother}.  From this, we deduce \eqref{ineq.mother} also for noncentral toral $x \in \so_8$ as in Example \ref{SO.tensor} and it follows that $\so_8$ acts generically freely on $V$.

\subsubsection*{Type $A_\ell$} Suppose now that $G = \SL_n / \mu_m$ with $\car k = 2$, so $m$ is even and $n \ge 4$.  As $m_p(G) = \dim L(\omega_2) = \binom{n}{2}$, we have 
\[
m_p(G) \, m(G) - b(G) = \frac12 \left( n^3 - 5n^2 + 4 \right).
\]
which is positive for $n \ge 5$.  So suppose further that $n = 4$, in which case $b(G) = 30$ and Table A.7 in \cite{luebeck} says that the smallest nontrivial restricted irreducible representations of $\SL_4$ have dimension 4 (the natural representation $k^4$ or its dual) or 6 ($\wedge^2 k^6$ with highest weight $\omega_2$), so $\la_0 = \omega_2$.  As $\SL_4$ does not act faithfully, up to graph automorphism $\la = 2^e \omega_1 + \omega_2$ for some $e \ge 1$.  These representations were handled in Example \ref{A3.eg}.  This completes the proof of Theorem \ref{MT} when $\car k$ is not special.$\hfill\qed$

\medskip

The proof of Theorem \ref{MT} in the remaining cases, when $\car k$ is special, will appear in part III, \cite{GG:special}.

\bibliographystyle{amsalpha}
\providecommand{\bysame}{\leavevmode\hbox to3em{\hrulefill}\thinspace}
\providecommand{\MR}{\relax\ifhmode\unskip\space\fi MR }
\providecommand{\MRhref}[2]{%
  \href{http://www.ams.org/mathscinet-getitem?mr=#1}{#2}
}
\providecommand{\href}[2]{#2}

\end{document}